\documentclass{amsart} 
\usepackage{amsmath,amssymb,amsthm,amsfonts,amsxtra,mathtools}
\usepackage{enumerate,verbatim,mathrsfs,comment,color,url}
\usepackage[all,2cell,ps,cmtip]{xy}
\usepackage[pagebackref]{hyperref}
\usepackage{graphicx}
\usepackage{verbatim}
\usepackage{cleveref}

\newcommand{\old}[1]{{\color{red} #1}}

\DeclareMathOperator{\ann}{ann}
\DeclareMathOperator{\Ass}{Ass}

\DeclareMathOperator{\cl}{Cl}

\DeclareMathOperator{\T}{T}
\DeclareMathOperator{\coker}{Coker}

\DeclareMathOperator{\depth}{depth}

\DeclareMathOperator{\End}{End}
\DeclareMathOperator{\Ext}{Ext}

\DeclareMathOperator{\gdim}{G-dim}

\DeclareMathOperator{\grade}{grade}
\DeclareMathOperator{\Hom}{Hom}

\DeclareMathOperator{\id}{id}
\DeclareMathOperator{\Image}{Image}

\DeclareMathOperator{\pd}{pd}

\DeclareMathOperator{\rank}{rank}

\DeclareMathOperator{\Spec}{Spec}
\DeclareMathOperator{\Soc}{Soc}
\DeclareMathOperator{\Supp}{Supp}
\DeclareMathOperator{\Tor}{Tor}

\DeclareMathOperator{\tr}{tr}
\DeclareMathOperator{\Tr}{Tr}

\DeclareMathOperator{\p}{\mathfrak p}
\DeclareMathOperator{\m}{\mathfrak m} 

\renewcommand{\ge}{\geqslant}
\renewcommand{\le}{\leqslant}

\newcommand{\fm}{\mathfrak{m}}
\newcommand{\fn}{\mathfrak{n}}
\newcommand{\fp}{\mathfrak{p}}
\newcommand{\fq}{\mathfrak{q}}
\newcommand{\fS}{\mathfrak{S}}
\renewcommand{\iff}{if and only if }

\newcommand{\syz}{\Omega}

\theoremstyle{plain}
\newtheorem{theorem}{Theorem}[section]
\newtheorem{lemma}[theorem]{Lemma}
\newtheorem{proposition}[theorem]{Proposition}
\newtheorem{corollary}[theorem]{Corollary}

\newtheorem{innercustomtheorem}{Theorem}
\newenvironment{customtheorem}[1]
  {\renewcommand\theinnercustomtheorem{#1}\innercustomtheorem}
  {\endinnercustomtheorem}

 \newtheorem{innercustomcorollary}{Corollary}
 \newenvironment{customcorollary}[1]
 {\renewcommand\theinnercustomcorollary{#1}\innercustomcorollary}
 {\endinnercustomcorollary}
  
\theoremstyle{definition}
\newtheorem{definition}[theorem]{Definition}
\newtheorem{conjecture}[theorem]{Conjecture}

\newtheorem{para}[theorem]{}
\newtheorem{question}[theorem]{Question}
\newtheorem{setup}[theorem]{Setup}

\theoremstyle{remark}
\newtheorem{remark}[theorem]{Remark}

\numberwithin{equation}{section}

\title[Finite homological dimension of Hom, vanishing of Ext]{Finite homological dimension of Hom, vanishing of Ext, and applications to divisor class group}

\author{Souvik Dey}
\address{Department of Mathematics, University of Arkansas, Fayetteville, AR 72701, U.S.A}
\email{souvikd@uark.edu, dey0976@gmail.com} 
\urladdr{\url{https://orcid.org/0000-0001-8265-3301}}

\author[Dipankar Ghosh]{Dipankar Ghosh}
\address{Department of Mathematics, Indian Institute of Technology Kharagpur, West Bengal - 721302, India}
\email{dipankar@maths.iitkgp.ac.in, dipug23@gmail.com}
\urladdr{\url{https://orcid.org/0000-0002-3773-4003}}

\thanks{Souvik Dey was partly supported by the Charles University Research Center program No.UNCE/SCI/022 and a grant GA CR 23-05148S from the Czech Science Foundation}    

\begin{document}  

\pagenumbering{arabic}
\thispagestyle{empty}
  
\begin{abstract}
 For finitely generated modules $M$ and $N $ over a commutative Noetherian local ring $R$, we give various sufficient criteria for detecting freeness of $M$ or $N$ via vanishing of some finitely many Ext modules $\Ext^i_R(M,N)$ and finiteness of certain homological dimension of $\Hom_R(M,N)$. Some of our results provide partial progress towards answering a question of Ghosh-Takahashi and also generalize their main results in many ways, for instance, by reducing the number of vanishing.
Certain special cases of our results allow us to address the Auslander-Reiten conjecture for modules whose (self-) dual has finite projective dimension. Along the way, we establish a new characterization of $I$-Ulrich modules of Dao-Maitra-Sridhar which we then apply to provide a negative answer to a question of Gheibi-Takahashi concerning characteristic modules. 
Among other techniques, we introduce and study certain generalizations of the notion of residually faithful modules of  Brennan-Vasconcelos and Goto-Kumashiro-Loan, which play a crucial role in our study. As some applications of our results, we provide affirmative answers to two questions raised by Tony Se on $n$-semidualizing modules. Namely, we show that over a local ring of depth $t$, every $(t-1)$-semidualizing module of finite G-dimension is free. Moreover, we establish that for normal domains which satisfy Serre's condition $(S_3)$ and are locally Gorenstein in codimension two, the class of $1$-semidualizing modules forms a subgroup of the divisor class group. These two groups coincide when, in addition, the ring is locally regular in codimension two.
\end{abstract}
 
\maketitle
\section{Introduction}
	
\begin{setup}\label{setup}
	Throughout this article, unless otherwise specified, $R$ is a commutative Noetherian ring, and all $R$-modules are assumed to be finitely generated.
\end{setup}

The aim of this paper is multi-fold. We begin by studying some variations of the celebrated Auslander-Reiten conjecture, and establish several criteria for finite projective or injective dimension of a module under certain vanishing of finitely many Ext modules and finite homological dimension of Hom. As application of our methods, we are also provide a negative answer to a question of Gheibi-Takahashi, \cite[Ques.~5.4]{GT24}. Our investigation naturally leads us to introduce various generalizations of the notion of residually faithful modules due to Brennan-Vasconcelo. These generalizations play a crucial role in proving many of our results regarding freeness criteria. Finally, we apply our techniques to understand the notion of $n$-semidualizing modules of Tony Se and their behaviour inside the divisor class group.
	
One of the most remarkable long-standing conjectures in homological commutative algebra is due to Auslander and Reiten. It gives a possible criteria for a module to be projective in terms of vanishing of certain Ext modules.
	
\begin{conjecture}\cite{AR75}\label{ARC}
	Let $M$ be an $R$-module. If $\Ext_R^i(M,M\oplus R) = 0 $ for all $i\ge 1$, then $M$ is projective.
\end{conjecture}


The conjecture is known to hold true for various special classes of rings, see, e.g., \cite[1.9]{ADS93}, \cite[Thm.~6.7]{DG23}, \cite[Thm.~4.1]{HSV04}, \cite[Thm.~0.1]{HL04}, \cite[Cor.~4]{Ar09}, \cite[Cor.~1.3]{KOT21},
\cite[Thm.~1.2]{Sh23}, \cite[1.2]{NS17}, 
\cite[Cor.~6.8]{NT20}, \cite{AINW} and \cite[Prop~5.10 and Thm.~6.7]{Ta}. However, the conjecture is broadly open even for Gorenstein rings. Among various classes of modules, the conjecture holds true when: (1)~$M$ has finite complete intersection dimension, cf.~\cite[Thm.~4.3]{AY98} and \cite[Thm.~4.2]{AB00}. (2) $M$ is a module over an Artinian local ring $(R,\fm)$ satisfying $\fm^2M=0$ \cite[Thm.~4.2]{HSV04} (see also \cite[Thm.~F]{DGS}). (3) $\Hom_R(M,R)$ or $\Hom_R(M,M)$ has finite injective dimension \cite[2.10.(2) and 2.15]{GT21}.

In the present study, we considerably strengthen \cite[2.10.(2) and 2.15]{GT21}, the main results in \cite{GT21}. Indeed, taking $N=R$   in \Cref{inj1} and \Cref{injcor}, one recovers and improves \cite[Cor.~2.10.(2)]{GT21}, while taking $N=M$ in \Cref{injimprov}, one recovers and improves \cite[Thm.~2.15]{GT21}. It is to be noted that in \Cref{injimprov}, one only needs the vanishing of $\Ext_R^{1 \le i \le t}(M,R)$, unlike in \cite[Thm.~2.15]{GT21}, where the vanishing of $\Ext_R^{1 \le i \le 2t+1}(M,R)$ is required, where $t = \depth(R)$. Moreover, both Theorems~\ref{inj1} and \ref{injimprov} provide partial positive answers to \cite[Ques.~2.9]{GT21}.



We establish several preliminary results in Section~\ref{sec:van-Ext-some-cons} to prove \Cref{inj1}, \Cref{injimprov} and \Cref{injcor}.  As a byproduct of our methods, we provide a new characterization of $I$-Ulrich modules of Dao-Maitra-Sridhar (see \Cref{ne2}) which we then use to provide a negative answer to a question of Gheibi-Takahashi, \cite[Ques.~5.4]{GT24}, on characteristic modules, see \Cref{para:char-mod}, \Cref{para:GT-ans} and \Cref{thm:negative-answer-GT}.


In \cite[Sec.~5]{BV01}, Brennan-Vasconcelos introduced and studied residually faithful modules, which were later studied by Goto-Kumashiro-Loan \cite[Sec.~3]{GKL19}. In \Cref{resfaith2}, we highly generalize their \cite[Thm.~3.9]{GKL19} even under weaker Ext-vanishing hypotheses. One of our results gives a freeness criteria of modules using this notion and the notion of generalized Ulrich modules as introduced in \cite[Defn.~1.2]{GOTWY}; see Proposition~\ref{resfaithul}. 
A crucial ingredient of the proof of \Cref{resfaithul} is a result comparing existence of nonzero free summand of a module and its dual, as shown in Lemma~\ref{freesumm}.

Taking cue from the definition \cite[3.1]{GKL19}, we introduce a few generalized versions of residually faithful modules, see Definition~\ref{defn:res-faithful}.
Using these notions, in Section~\ref{sec:res-faithful}, we prove the following theorem, which 
helps us to address Conjecture~\ref{ARC} when projective dimension of $\Hom_R(M,M)$ is finite.

\begin{customtheorem}{\ref{thm:pd-Hom-finite-gdim-freeness-criteria}}
    Let $R$ be a local ring of depth $t$. Let $C$ be a semidualizing $R$-module $($e.g., $C=R$$)$. Let $M$ and $N$ be $R$-modules such that $\Hom_R(M,N)$ is nonzero,
    \[
        \Ext_R^{1\le i \le t-1}(M,N) = 0 \;\mbox{ and } \;\pd_R\big(\Hom_R(C,\Hom_R(M,N))\big) < \infty.
    \]
    Let $N \cong F \oplus \syz_R(Y)$ for some $R$-modules $Y$ and $F$ each of depth at least $t$ $($e.g., if $N$ is $(t+1)$-torsion-free, or if $\gdim_R(N)=0$$)$. Then, $N\cong F$, $C\cong R$, and $M\cong t(M)\oplus G$ for some nonzero free $R$-module $G$, where $t(M)$ denotes the torsion submodule of $M$.
\end{customtheorem}

Specializing Theorem~\ref{thm:pd-Hom-finite-gdim-freeness-criteria} to $C=R$ or $C=\omega$, a canonical module of $R$, and combining with other results, we are able to prove some freeness criteria in \Cref{thm:Hom-pd-id} that complement Theorems~\ref{inj1} and \ref{injimprov}.
The results (1) and (2) in Theorem~\ref{thm:Hom-pd-id} are partly motivated by a special case of Conjecture~\ref{ARC} which is known as Tachikawa's conjecture; see \cite[pp.~219]{ABS05}.

\begin{conjecture}[Tachikawa]\label{Tachikawa}
    A Cohen-Macaulay local ring $R$ with a canonical module $\omega$ is Gorenstein (equivalently, $\omega$ is free) if $\Ext_R^i(\omega,R)=0$ for all $i \ge 1$.
\end{conjecture}

For a canonical module $\omega$ of $R$, one has that $ \Ext_R^i(\omega,\omega) = 0 $ for all $i \ge 1$. Moreover $ \Hom_R(\omega,\omega) \cong R $. So we are interested whether Conjecture~\ref{ARC} holds true when $\Hom_R(M,M)$ is free or more generally when $\pd_R(\Hom_R(M,M))$ is finite. This is quite difficult, as it will confirm Tachikawa's conjecture, which is widely open. 
However, as a consequence of Theorem~\ref{thm:Hom-pd-id}, we obtain the following.

\begin{customcorollary}{\ref{cor:ARC-pd-Hom}}
	The Auslander-Reiten conjecture holds true for a module $M$ over a commutative Noetherian ring $R$ in each of the following cases:
	\begin{enumerate}[\rm (1)]
		\item $ \pd_R\left( \Hom_R(M,R) \right) < \infty $.
		\item $ \pd_R\left( \Hom_R(M,M) \right) < \infty $ and $\gdim_R(M) < \infty$.
	\end{enumerate}
\end{customcorollary}

Motivated by Conjecture~\ref{Tachikawa}, we also prove Theorem~\ref{thm:pd-Hom-Ext-Tr}, which particularly provides a few criteria for a semidualizing (cf.~\ref{para:Semidualizing}) module $M$ to be free in terms of vanishing of certain $\Ext_R^i(\Tr(M),R)$ or $\Ext_R^i(M^*,R)$, where $\Tr(M)$ and $M^*$ denote the Auslander transpose (cf.~\ref{para:Tr-M}) and the dual of $M$ respectively.

In the last section, we address two questions raised by Tony Se in \cite{Se}. Both the questions are on $n$-semidualizing modules (cf.~Definition~\ref{defn:n-semidualizing}).

\begin{question}\cite[2.10]{Se}\label{ques1:Tony-Se}
Over a Gorenstein ring of finite dimension $d$, is every $(d-1)$-semidualizing module exactly semidualizing?
\end{question}

\begin{question}\cite[5.2]{Se}\label{ques2:Tony-Se}
If $R$ is a normal domain, then whether the set $\fS_0^1(R)$ of isomorphism classes of $1$-semidualizing modules is a subgroup of the divisor class group $\cl(R)$ (cf.~\ref{para:class-group}).
\end{question} 

We provide a complete answer to Question~\ref{ques1:Tony-Se} by proving a stronger result that over a local ring of depth $t$, every $(t-1)$-semidualizing module of finite G-dimension is in fact projective, see Theorem~\ref{mainsemi} and Corollary~\ref{qa}. Note that a projective module $M$ satisfying $\Hom_R(M,M)\cong R$ is semidualizing.

Related to Question~\ref{ques2:Tony-Se}, the relations between $\fS_0^1(R)$ and $\cl(R)$ are investigated in \cite{Se}. If $R$ is a normal domain (equivalently, if $R$ is domain and it satisfies $(S_2)$ and $(R_1)$, cf.~\ref{para-Sn} and \ref{para-Rn}), then as a set $\fS_0^1(R)$ is contained in $\cl(R)$, however there are examples where both the sets are identical, and examples (of Gorenstein normal domains) where both the sets are different, see \ref{para:S1-subset-cl-R} for the detailed references. In the present study, we prove the following results towards Question~\ref{ques2:Tony-Se}.
\begin{customtheorem}{\ref{thm:S1-Cl}}
Let $R$ be a normal domain satisfying $(S_3)$.
\begin{enumerate}[\rm (1)]
    \item If $R$ is locally Gorenstein in codimension $2$, then $\fS_0^1(R)$ is a subgroup of $\cl(R)$.
    \item If $R$ satisfies $(R_2)$, then $ \fS^1_0(R) = \cl(R) $.
\end{enumerate}
\end{customtheorem}

Note that the class of rings as in Theorem~\ref{thm:S1-Cl}.(2) contains all determinantal rings with coefficients in a regular local ring, see \ref{para:Deter-rings} and Corollary~\ref{cor:on-deter-rings}. In particular, Theorem~\ref{thm:S1-Cl}.(2) highly generalizes \cite[Thm.~4.13]{Se}, the main result in \cite{Se}. 

The organization of the paper is as follows. Section~\ref{sec:prel} records basic notations, definitions and some preliminary observations. Section~\ref{sec:van-Ext-some-cons} proves several preparatory results and certain applications of them, including a negative answer to \cite[Ques.~5.4]{GT24}. Most of the preparatory results in Section~\ref{sec:van-Ext-some-cons} are heavily used in the subsequent sections. Section~\ref{sec:id-hom} is mainly devoted towards proving \Cref{inj1},  \Cref{injimprov} and \Cref{injcor}. Section~\ref{sec:res-faithful} contains Proposition~\ref{resfaith2} and Theorem~\ref{thm:pd-Hom-finite-gdim-freeness-criteria}, the former being  a generalization of \cite[Thm.~3.9]{GKL19} due to  Goto-Kumashiro-Loan. Section~\ref{sec:pd-hom} consists of Theorems~\ref{thm:pd-Hom-Ext-Tr} and \ref{thm:Hom-pd-id} along with other results. Section~\ref{clgroup} provides various applications to $n$-semidualizing modules and divisor class  group, culminating in Theorem~\ref{mainsemi}, Theorem~\ref{thm:S1-Cl} and  Corollary~\ref{cor:on-deter-rings}, which particularly answer two questions in \cite{Se}.



\section{Preliminaries}\label{sec:prel}
Here we recall a few notations, terminologies, definitions and basic known results that are useful in this article. Throughout, by CM and MCM, we mean Cohen-Macaulay and maximal Cohen-Macaulay respectively. Let $M$ be an $R$-module. Set $M^* := \Hom_R(M,R)$. Let $t(M)$ be the torsion submodule of $M$. The length of $M$ is denoted by $\ell_R(M)$.
For a local ring $(R,\m,k)$, the $n$th Bass number is defined to be $\mu^n_R(M) := \rank_k \left(\Ext^n_R(k,M)\right)$. Moreover, in the local setup, $\mu_R(M)$ (or $\mu(M)$ when the base ring is clear) denotes the minimal number of generators of $M$ as an $R$-module. For an $\fm$-primary ideal $I$ of $R$, the Hilbert-Samuel multiplicity of $M$ with respect to $I$ is denoted by $e_R(I,M)$.

\begin{para}\label{para:Tr-M}
    For an $R$-module $M$, consider a projective presentation $F_1 \stackrel{\eta}{\to} F_0 \to M \to 0$ of $M$. The $($first$)$ syzygy and Auslander transpose of $M$ are defined to be $\syz_R(M) := \Image(\eta) $ and $\Tr(M) := \coker(\eta^*)$ respectively. When the base ring is clear, we drop the subscript from $\syz_R(M)$. Set $\syz^1(M)=\syz(M)$, and inductively $\syz^n(M)=\syz^1(\syz^{n-1}(M))$ for all $n\ge 2$. These modules are uniquely determined by $M$ up to projective summands.
    By the definition of $\Tr(M)$, there is an exact sequence
\begin{equation}\label{eqn:es-M*-Tr-M}
    0 \to M^* \to F_0^* \stackrel{\eta^*}{\longrightarrow} F_1^* \to \Tr(M)\to 0.
\end{equation}
It is well known that $M$ and $\Tr(\Tr M)$ are stably isomorphic (i.e., these two modules are isomorphic up to projective summands). For further details on this topic, see \cite{AB69}. When $R$ is local, $\syz^n(M)$ is only defined by considering minimal free resolution of $M$. We call a module $X$ to be an $n$-th syzygy module if $X$ is of the form $\syz^n(N)$ for some module $N$. Consequently, when $R$ is local, any $n$-th syzygy module is of the form $F\oplus \syz^n(N)$ for some modules $F$ and $N$ with $F$ free.
\end{para}

\begin{para}\label{para:Ext-syz}
    Ler $R$ be local. Let $n\ge 0$ be an integer. If $\Ext_R^{1\le i \le n-1}(M, R)=0$, then $M^*\cong F\oplus \syz_R^{n+1}(L)$ for some $R$-modules $F$ and $L$ with $F$ free. Indeed, consider a minimal free resolution $\mathbb{F}_M : \cdots \to F_n \to F_{n-1}\to\cdots\to F_1\to F_0 \to 0$ of $M$. Since $\Ext_R^{1\le i \le n-1}(M, R)=0$, dualizing $\mathbb{F}_M$, one obtains an exact sequence $0 \to M^* \to F_0 \to F_1 \to\cdots\to F_{n-1} \to F_n \to L \to 0$ for some $R$-module $L$. Hence the assertion follows.
\end{para}


\begin{para}\label{para-Sn}
An $R$-module $M$ is said to satisfy the Serre condition $(S_n)$ (resp., $(\widetilde S_n$) if
\begin{center}
$\depth_{R_{\p}}(M_{\p}) \ge \inf\{n,\dim(R_{\p})\}$ (resp., $\depth_{R_{\p}}(M_{\p}) \ge \inf\{n,\depth(R_{\p})\}$)
\end{center}
for all $\p \in \Spec(R).$ Note that if $R$ satisfies $(S_n)$, then there is no difference between the modules satisfying  $(S_n)$ and those satisfying $(\widetilde S_n)$.
\end{para}

\begin{para}\label{para:property-in-codim}
    An $R$-module $M$ is said to satisfy a property (P) in codimension $n$ (resp., in co-depth $n$) if the $R_{\p}$-module $M_{\p}$ satisfies the property (P) for every prime ideal $\p$ with $\dim(R_{\p})\le n$ (resp., $\depth(R_{\p})\le n$). Similarly, one defines a property of the ring $R$ in codimension $n$ (resp., in co-depth $n$).
\end{para}

\begin{para}\label{para-Rn}
    A ring $R$ is said to satisfy the Serre condition $(R_n)$ if $R$ is locally regular in codimension $n$.
\end{para}

\begin{para}\label{para:-n-torsion-free-nth-syz}
    Let $n\ge 0$ be an integer. Recall from \cite[Defn.~2.15]{AB69} that an $R$-module $M$ is said to be $n$-torsion-free if $\Ext_R^{1\le i \le n}(\Tr M, R)=0$. So, in this terminology, $M$ is called torsionless (resp., reflexive) if it is $1$-torsion-free (resp., $2$-torsion-free), see, e.g., \cite[1.4.21]{BH93}. If $M$ is $n$-torsion-free, then $M $ is an $n$-th syzygy module (\cite[Thm.~2.17]{AB69}), and consequently, satisfies $(\widetilde S_n)$ (cf.~\cite[1.3.7]{BH93}). Thus, if $R$ is local, then every $n$-torsion-free $R$-module $M$ can be written as $M \cong F \oplus \syz^n_R(L) $ for some $R$-modules $L$ and $F$ with $F$ free.
\end{para}

\begin{para}\label{para:G-dim-n-torsion-free}
    An $R$-module $M$ is said to be G-projective or totally-reflexive if both $M$ and $\Tr(M)$ are $n$-torsion-free for every $n\ge 1$. The G-dimension of $M$, denoted $\gdim_R(M)$, is defined to be the infimum of non-negative integers $n$ such that there exists an exact sequence $0 \to G_n \to G_{n-1} \to \cdots \to G_0 \to M \to 0$ of $R$-modules, where each $G_i$ is G-projective.
\end{para}

\begin{para}\label{para:Semidualizing}
An $R$-module $M$ is called semidualizing if $\Ext_R^j(M,M)=0$ for all $j\ge 1$ and $\Hom_R(M,M)\cong R$. A free module of rank $1$ and every canonical module are examples of semidualizing modules.
\end{para}

The notion of $n$-semidualizing modules was introduced by Takahashi in \cite[Defn.~2.3]{Tak07}. Later, it was slightly modified by Tony Se in \cite[Defn.~2.1]{Se}.

\begin{definition}\cite[Defn.~2.1]{Se}\label{defn:n-semidualizing}
    Let $n\ge 0$. An $R$-module $M$ is said to be $n$-semidualizing if $\Hom_R(M,M)\cong R$ and $\Ext^i_R(M,M)=0$ for all $1\le i\le n$. The set of all isomorphism classes of $n$-semidualizing modules of $R$ is denoted by $\fS^n_0(R)$.
\end{definition}

Next, we record some possibly known preliminary observations on how length and multiplicity behave with respect to local flat base change. These will be used in connection with how an `Ulrich like property' of modules behave with respect to considering Hom under vanishing of certain Ext.

\begin{lemma}\label{lem:JS-J-reg-gen}
    Let $(R,\fm) \longrightarrow (S,\fn)$ be a flat local ring homomorphism. Let $J$ be an ideal of $R$, and $I := JS$ be the extension of $J$ in $S$. Let $M$ be an $R$-module. Then
    \begin{enumerate}[\rm (1)]
        \item $J$ is generated by an $R$-regular sequence \iff $I$ is generated by an $S$-regular sequence.
        \item $\ell_S(S\otimes_R M) = \ell_R(M)\ell_S(S/\fm S)$.
        \item Let $\fm S = \fn$. Then
        \begin{enumerate}[\rm (a)]
            \item $\ell_R(M/JM) = \ell_S((S\otimes_R M)/I(S\otimes_R M))$.
            \item $J$ is $\fm$-primary \iff $I$ is $\fn$-primary.
            \item When $J$ is $\fm$-primary, $e_R(J,M)=e_S(I,S\otimes_R M)$.
        \end{enumerate}
        
    \end{enumerate}  
\end{lemma}

\begin{proof}
    (1) For the only if part, let $J= \langle {\bf x} \rangle$, where ${\bf x} = x_1,\ldots,x_r$ is an $R$-regular sequence. Then $JS = {\bf x}S$. Since $R \to S$ is flat, by induction on $r$, it can be verified that ${\bf x}$ is also regular on $S$. It remains to prove the if part Denote $I=JS$. Note that $I\cong J\otimes_R S$. As $R \to S$ is faithfully flat, if $I=0$, then $J=0$. So we may assume that $I\neq 0$. Since $I$ is generated by an $S$-regular sequence, $\pd_S(I)<\infty$. Moreover, $I/I^2$ is a free $S/I$-module, see, e.g., 
    \cite[1.1.8]{BH93}.
    Since $R \to S$ is a local flat homomorphism, tensoring a minimal free resolution of $J$ over $R$ by $S$, one obtains that $\pd_R(J) = \pd_S(I)<\infty$. Let $I/I^2 \cong (S/I)^{\oplus n}$ for some $n\ge 1$. Then it follows that
    \[
        \dfrac{J}{J^2} \otimes_R S \cong \dfrac{J\otimes_R S}{J^2\otimes_R S} \cong \dfrac{JS}{J^2S} \cong \dfrac{I}{I^2} \cong \left( \dfrac{S}{I} \right)^{\oplus n} \cong \left( \dfrac{R}{J} \otimes_R S \right)^{\oplus n} \cong \left( \dfrac{R}{J} \right)^{\oplus n} \otimes_R S.
    \]
    Therefore, by \cite[ Prop.~2.5.8]{Gro} , $J/J^2 \cong (R/J)^{\oplus n}$. Hence, by \cite[2.2.8]{BH93}, $J$ is generated by an $R$-regular sequence.

    (2) This equality is well known, see, e.g., \cite[Tag~02M1]{Stacks} or \cite[1.2.25]{BH93}.

    (3) The proof of (a) is same as that of \cite[4.3]{ddd}. We give the details for the convenience of the reader. Since $\fm S = \fn$, one has that
    \begin{align*}
        \ell_R(M/JM) &= \ell_R(M/JM)\ell_S(S/\fm S) = \ell_S\big(S\otimes_R (M/JM)\big) \quad \mbox{[appying (2) on $M/JM$]}\\
        &= \ell_S\big((S\otimes_R M)/(S\otimes_R (JM))\big) = \ell_S((S\otimes_R M)/I(S\otimes_R M)),
    \end{align*}
    where the last equality follows as $S\otimes_R (JM) = (JS)(S\otimes_R M)$ when identified as submodules of $S\otimes_R M$, see, e.g., \cite[5.2.5(1)]{ddd}. This proves (a). Considering $M=R$ in (a), one obtains (b). For (c), note that $\dim_R(M) = \dim_S(S\otimes_R M)$ (cf.~\cite[Thm.~A.11(b)]{BH93}). Moreover $J^nS = I^n$. Therefore, in view of (a),
    \begin{align*}
        e_R(J,M) &= \lim_{n\to \infty} \dfrac{\dim(M)!}{n^{\dim(M)}}\;\ell_R\Big(\dfrac{M}{J^{n+1}M}\Big) \\
        &= \lim_{n\to \infty} \dfrac{\dim(S\otimes_R M)!}{n^{\dim(S\otimes_R M)}} \; \ell_S\left(\dfrac{S\otimes_R M}{I^{n+1}(S\otimes_R M)}\right) = e_S(I,S\otimes_R M).
    \end{align*}
    This completes the proof of the lemma.
\end{proof}

We recall the notion of Ulrich modules (with respect to an $\fm$-primary ideal).

\begin{definition}\label{defn:Ulrich}
    Let $(R,\fm)$ be a CM local ring, and $J$ be an $\fm$-primary ideal of $R$.
    \begin{enumerate}[\rm (1)]
        \item \cite[Defn.~1.2]{GOTWY} An MCM $R$-module $M$ is said to be Ulrich with respect to $J$ if $e_R(J,M)=\ell_R(M/JM)$ and $M/JM$ is free over $R/J$.
        \item \cite[pp.~183]{BHU} An $R$-module is said to be Ulrich if it is Ulrich with respect to $\fm$.
    \end{enumerate}
\end{definition}

Lemma~\ref{lem:JM=xM} is closely related to the notion of Ulrich modules, and will be subsequently used in Proposition~\ref{mainul} and Lemma~\ref{lem:J=x}. 

\begin{lemma}\label{lem:JM=xM}
    Let $(R,\m)$ be a CM local ring of infinite residue field, $M$ be an MCM $R$-module, and $J$ be an $\m$-primary ideal of $R$. Then, the following statements are equivalent:
    \begin{enumerate}[\rm (1)]
        \item $e_R(J,M)=\ell_R(M/JM)$. 
        \item $JM=(\mathbf x )M$ for some $R$-regular sequence $\mathbf x$ such that $(\mathbf x)$ is a reduction ideal of $J$.
        \item $JM=(\mathbf x )M$ for every $R$-regular sequence $\mathbf x$ such that $(\mathbf x)$ is a reduction ideal of $J$.
    \end{enumerate}
\end{lemma}

\begin{proof}
    Let ${\bf x}$ be an $R$-regular sequence such that $(\mathbf x)$ is a reduction ideal of $J$. Such a sequence always exists as $R$ is CM, $J$ is $\m$-primary, and the residue field of $R$ is infinite, see, e.g., \cite[4.6.10]{BH93}. So the implication (3) $\Rightarrow$ (2) is trivial. For the other implications, note that $e({\bf x}, M)=e_R(J, M)$ (cf.~\cite[4.6.5]{BH93}). As $M$ is MCM, the $R$-regular sequence ${\bf x}$ is also regular on $M$. Hence it follows from \cite[1.1.8]{BH93} that
    \[
        \dfrac{({\bf x})^nM}{({\bf x})^{n+1}M} \cong \left(\dfrac{M}{({\bf x})M}\right)^{\bigoplus {n+d-1\choose d-1}}
    \]
    which implies that $e_R(({\bf x}), M) = \ell_R(M/({\bf x})M)$. Therefore
    \[
        e_R(J, M) = e_R(({\bf x}), M) = \ell_R(M/({\bf x})M).
    \]
    This yields that $e_R(J, M) = \ell_R(M/JM)$ \iff $JM=({\bf x})M$ as $({\bf x})\subseteq J$. Consequently, (1) $\Rightarrow$ (3) and (2) $\Rightarrow$ (1).
\end{proof}

The notion of trace ideal of a module is used in Lemmas~\ref{tracereg} and \ref{freesumm}.

\begin{para}\label{para:trace}
    The trace ideal of an $R$-module $M$, denoted $\tr_R(M)$, is defined to be the ideal $\sum f(M)$, where $f$ varies in $M^*$. Note that $\tr_R(M)$ is same as the image of the natural map $\Phi : M\otimes_R M^* \to R$ defined by $\Phi(x,f) = f(x)$ for all $x\in M$ and $f\in M^*$.
\end{para}

\section{Vanishing of Ext and some consequences}\label{sec:van-Ext-some-cons}

In this section, we prove certain fundamental results, some of which are useful in the latter sections. 

\begin{lemma}\label{extres}
Let $M$ and $N$ be $R$-modules such that $\Ext_R^{1\le i \le n-1}(M,N)=0$ for some positive integer $n$. The following statements hold true:
\begin{enumerate}[\rm(1)] 
    \item Let ${\bf x} := x_1,\dots,x_n$ be an $N$-regular sequence of length $n$. Then
    \begin{enumerate}[\rm (i)]
        \item ${\bf x}$ is also $\Hom_R(M,N)$-regular.
        \item $\Hom_R(M,N)/{\bf x} \Hom_R(M,N)$ embeds in $\Hom_R(M, N/{\bf x} N)$.
        \item In addition, if $\Ext_R^n(M,N)=0$, then
        \[
        \Hom_R(M,N)/{\bf x} \Hom_R(M,N) \cong \Hom_R(M, N/{\bf x} N) \cong \Hom_{R/{\bf x}R}(M/{\bf x} M, N/{\bf x} N).
        \]
    \end{enumerate}
    \item $\depth_{R_{\p}}(\Hom_{R_{\p}}(M_{\p},N_{\p})) \ge \inf\{n, \depth_{R_{\p}}(N_{\p})\}$ for all $\p \in \Spec(R)$. 
\end{enumerate}    
\end{lemma}   

\begin{proof}
(1) We prove this by induction on $n$. First assume that $n=1$. The exact sequence $0\to N \xrightarrow{x_1} N \to N/x_1 N \to 0$ induces a long exact sequence
\begin{align}\label{les-Ext-N-mod-x1}
    0 & \to \Hom_R(M,N)\xrightarrow{x_1}\Hom_R(M,N) \to \Hom_R(M,N/x_1 N) \\
    & \to \Ext_R^1(M,N) \xrightarrow{x_1} \Ext_R^1(M,N) \to \Ext_R^1(M,N/x_1 N) \to \cdots \nonumber
\end{align}
showing that $x_1$ is $\Hom_R(M,N)$-regular, and there is an embedding
\begin{equation}\label{embedding-hom}
    \Hom_R(M,N)/x_1\Hom_R(M,N) \xhookrightarrow{\;\;\;\;\;\;\;\;\;} \Hom_R(M,N/x_1 N).
\end{equation}
In addition, if $\Ext_R^1(M,N)=0$, then the embedding in \eqref{embedding-hom} is an isomorphism. This proves the base case. Next assume that $n\ge 2$. Let $N':=N/x_1N$. Since $\Ext_R^{1\le i \le n-1}(M,N)=0$ and $n-1\ge 1$, in view of \eqref{les-Ext-N-mod-x1}, one obtains that
\begin{equation}\label{Hom-mod-x_1}
    \Hom_R(M,N)/x_1\Hom_R(M,N)\cong \Hom_R(M,N'),
\end{equation}
and $\Ext^{1\le i\le n-2}_R(M,N')=0$. As ${\bf x'} := x_2,\dots,x_n$ is $N'$-regular, by the induction hypothesis, ${\bf x'}$ is regular on $\Hom_R(M,N')$, and
\begin{equation}\label{embedding-hom-2}
    \Hom_R(M,N')/{\bf x'}\Hom_R(M,N') \xhookrightarrow{\;\;\;\;\;\;\;\;\;} \Hom_R(M,N'/{\bf x'}N').
\end{equation}
When $\Ext_R^{1 \le i \le n}(M,N)=0$, from \eqref{les-Ext-N-mod-x1}, it follows that $\Ext^{1 \le i \le n-1}_R(M,N')=0$, and hence by the induction hypothesis, the embedding in \eqref{embedding-hom-2} is an isomorphism. Therefore, in view of the isomorphism \eqref{Hom-mod-x_1}, one derives that ${\bf x}$ is $\Hom_R(M,N)$-regular, $\Hom_R(M,N)/{\bf x} \Hom_R(M,N)$ embeds in $\Hom_R(M, N/{\bf x} N)$, and these two modules are isomorphic under the extra condition that $\Ext_R^n(M,N)=0$. The last isomorphism in 1.(iii) follows from the general observation that for any $R$-modules $V,W$ and any ideal $I$ of $R$, it holds that $\Hom_R(V, W/IW) \cong \Hom_{R/I}(V/IV, W/IW)$.

(2) Let $\p \in \Spec(R)$. Since $\Ext_R^{1\le i \le n-1}(M,N)=0$, it follows that $\Ext_{R_{\p}}^{1\le i \le n-1}(M_{\p}, N_{\p}) =0$. In order to establish the desired inequality, replacing $R_{\p}$, $M_{\p}$ and $N_{\p}$ by $R$, $M$ and $N$ respectively, we may assume that $R$ is local.
Set $m:=\inf \{n, \depth_R(N) \}$. Consider a regular sequence ${\bf y} = y_1,\ldots,y_m$ on $N$.  Then, by (1).(i), ${\bf y}$ is also $\Hom_R(M,N)$-regular, hence the claim follows.
\end{proof}

Under the additional conditions that $J=\fm$ and $M$ being MCM, the following result is shown in \cite[Prop.~4.1]{KT}.

\begin{proposition}\label{mainul}
Let $(R,\m)$ be a local CM ring of dimension $d$, and $J$ be an $\m$-primary ideal. Assume that there exist $R$-modules $M$ and $N$ such that $N$ is MCM, $e_R(J,N)=\ell_R(N/JN)$, and $\Ext_R^{1\le i\le d-1}(M,N)=0$. Then $\Hom_R(M,N)$ is MCM and satisfies
\[
    e_R\big(J,\Hom_R(M,N)\big) = \ell_R\big(\Hom_R(M,N)/J\Hom_R(M,N)\big).
\]
\end{proposition} 

\begin{proof}
By Proposition~\ref{extres}, $\Hom_R(M,N)$ is MCM. In view of Lemma~\ref{lem:JS-J-reg-gen}, if necessary, replacing $R$ by $R[X]_{\fm R[X]}$, we may assume that $R/\fm$ is infinite. Then, since $e_R(J,N)=\ell_R(N/JN)$, by Lemma~\ref{lem:JM=xM}, there exists an $R$-regular sequence $\mathbf x$ such that $(\mathbf x)$ is a reduction ideal of $J$, and $JN=(\mathbf x )N$. Consequently, $\Hom_R(M,N/{\bf x}N)$ is annihilated by $J$. So, by Proposition~\ref{extres}.(1).(ii), $\Hom_R(M,N)/{\bf x} \Hom_R(M,N)$ is annihilated by $J$. Thus, $J\Hom_R(M,N) = {\bf x} \Hom_R(M,N)$. Therefore, again by Lemma~\ref{lem:JM=xM}, one obtains the desired equality.
\end{proof}

The following proposition provides a freeness criteria for a module in terms of certain Ext vanishing which will be very useful in the latter sections.

\begin{proposition}\label{propnew}
Let $M$ and $N$ be nonzero modules over a local ring $R$. Set $s := \depth(N)$. If $\Ext_R^{1\le i\le s}(M,N) = 0 = \Ext_R^{1\le i\le s+1}\big(\Tr(M), \Hom_R(M,N)\big) $, then $M$ is free.
\end{proposition}

\begin{proof}
Let $ {\bf x} := x_1,\ldots,x_s $ be an $N$-regular sequence. Then, by Lemma~\ref{extres}.(1), ${\bf x}$ is also $\Hom_R(M,N)$-regular, and $\Hom_R(M,N)/{\bf x} \Hom_R(M,N) \cong \Hom_R(M,N/{\bf x} N)$. Since ${\bf x}$ is $\Hom_R(M,N)$-regular and $\Ext_R^{1\le i\le s+1}\big(\Tr(M), \Hom_R(M,N)\big)=0$, one derives that
$$\Ext_R^{1}\big(\Tr(M), \Hom_R(M,N)/{\bf x} \Hom_R(M,N)\big) = 0.$$
Hence $\Ext^1_R\big(\Tr(M), \Hom_R(M,N/{\bf x} N)
\big)=0$. Since $\depth(N/{\bf x} N)=0$, the claim now follows from \cite[Prop.~3.3.(2)]{KOT21}.  
\end{proof}

In the following, we repeatedly use that if $\mathbf x :=x_1,\ldots,x_n$ is an $M$-regular sequence, then $\Hom_R(k,M/{\mathbf x} M) \cong \Ext_R^{n}(k,M)$ (cf.~\cite[1.2.4]{BH93}). The first part of Proposition~\ref{prop:GT} gives an alternative proof of \cite[Lem.~2.1]{gt} without using spectral sequences. 

\begin{proposition}\label{prop:GT} Let $(R,\m,k)$ be a local ring. Let $M$ and $N$ be $R$-modules, and $n$ be an integer such that $0\le n \le \depth(N)$ and $\Ext_R^{1\le i \le n}(M,N)=0$. Then, the following hold true
\begin{enumerate}[\rm(1)]
    \item $ \mu_R(M) \mu_R^n(N) = \mu_R^n(\Hom_R(M,N)) $.
    \item If $n=\depth(N)$, and the modules $M$ and $N$ are nonzero, then $\depth(\Hom_R(M,N))=n$.    
\end{enumerate}
\end{proposition}

\begin{proof}
We may assume that $M$ and $N$ are nonzero. Let ${\mathbf x} =x_1,\ldots,x_n$ be an $N$-regular sequence. Then, by Lemma~\ref{extres}.(1), $\mathbf x$ is also $\Hom_R(M,N)$-regular, and $\Hom_R(M,N)/{\bf x} \Hom_R(M,N) \cong \Hom_R(M, N/{\bf x} N)$. Therefore
  \begin{align*}
      \mu_R^n(\Hom_R(M,N)) &= \rank_k\big( \Ext_R^n(k, \Hom_R(M,N)) \big) \\
      & = \rank_k\big(\Hom_R\big(k, \Hom_R(M,N)/{\bf x} \Hom_R(M,N)\big) \big) \\
      &=\rank_k\big(\Hom_R(k, \Hom_R(M, N/{\bf x} N)\big)\\
      &= \rank_k\big(\Hom_R(k\otimes_R M, N/{\mathbf x} N)\big) \\
      & =\mu_R(M)\rank_k \big(\Hom_R(k, N/{\mathbf x} N)\big)\\
      & =\mu_R(M)\rank_k \big(\Ext_R^n(k, N)\big) =\mu_R(M)\mu^n_R(N).
  \end{align*}
  It remains to prove the second part. Note that $\depth(\Hom_R(M,N))\ge n$. Since $M$ and $N$ are nonzero, and $n=\depth(N)$, it follows that $\mu_R(M)\mu^n_R(M)\neq 0$. Hence, by part one, $\mu_R^n(\Hom_R(M,N))\neq 0$. Therefore $\depth(\Hom_R(M,N))\le n$, proving our claim.
\end{proof}

As a consequence of Proposition~\ref{prop:GT}, we obtain  some interesting freeness criteria via Ext vanishing. The first one is a criterion for a module finite  birational extension of a local ring to be trivial one in terms of existence of certain torsion-free module. 

\begin{corollary}
Let $R$ be a local ring. Consider a module finite ring extension $R\subseteq S$ such that $S \subseteq Q(R)$, where $Q(R)$ denotes the total ring of fractions. If there exists a nonzero torsion-free $S$-module $N$ such that $\Ext_R^{1\le i \le \depth_R(N)}(S,N)=0$, then $S=R$.  
\end{corollary}  

\begin{proof}
Let $N$ be a nonzero torsion-free $S$-module such that $\Ext_R^{1\le i \le n}(S,N)=0$, where $n:=\depth(N)$. In view of \cite[Lem.~5.2.9]{ddd}, $\Hom_R(S,N)=\Hom_S(S,N)\cong N$. Since $\mu_R^n(N)\neq 0$, by Proposition~\ref{prop:GT}.(1), one obtains that $\mu_R(S)=1$. So $S\cong R/I$ for some ideal $I$ of $R$. Since $S$ is faithful as an $R$-module, it follows that $I=0$, i.e., $S\cong R$ as $R$-modules. Thus $S=Rx$ for some $x\in Q(R)$. In particular, $1=rx$ for some $r\in R$. Hence $x$ is a non-zero-divisor in $Q(R)$. Therefore $S=(S:_{Q(R)} S)=(Rx :_{Q(R)} Rx)=(R:_{Q(R)} R)=R$.
\end{proof}

As another consequence of Proposition~\ref{prop:GT}, we give a freeness criteria for an $\m$-primary ideal in a CM local ring of dimension $1$. For this, we first record two preparatory lemmas.

\begin{lemma}\label{n1}
Let $M$ and $N$ be $R$-modules, and $I$ be an ideal of $R$. Suppose $N$ is torsion-free. If $M$ is either locally free on $\Ass(R)$, or $I$ contains a non-zero-divisor, then 
    $\Hom_R(I, \Hom_R(M,N))\cong \Hom_R(IM,N)$.  
\end{lemma}

\begin{proof}
The short exact sequence $0\to I\to R\to R/I\to 0$ induces another exact sequence $0\to \Tor^R_1(R/I, M)\to I\otimes_R M\to IM\to 0$. If $M$ is locally free on $\Ass(R)$, then  $\Tor^R_1(R/I, M)$ is a torsion module. In the other case, when $I$ contains a non-zero-divisor, the module $\Tor^R_1(R/I, M)$ (being annihilated by $I$) is torsion. Therefore, since $N$ is torsion-free, in any case, one obtains that $\Hom_R(\Tor^R_1(R/I, M), N)=0$. Hence, applying $\Hom_R(-,N)$ to the above short exact sequence, we get $\Hom_R(IM,N)\cong \Hom_R(I\otimes_R M,N)\cong \Hom_R(I, \Hom_R(M,N))$, where the last isomorphism follows from tensor-hom adjunction. 
\end{proof}

We recall from \cite[Def.~4.1 and Thm.~4.6]{dms} the definition of $I$-Ulrich modules.

\begin{definition}
    Let $I$ be an $\m$-primary ideal of a CM local ring $(R,\m)$ of dimension $1$. Let $M$ be an MCM $R$-module. Then, $M$ is called $I$-Ulrich if $IM\cong M$.
\end{definition}

Our next \Cref{ne2} gives a new characterization of $I$-Ulrich modules.  

\begin{lemma}\label{ne2}
Let $(R,\m)$ be a CM local ring of dimension $1$. Let $M$ be an MCM $R$-module, and $I$ be an $\m$-primary ideal of $R$. Then $IM\cong M$ \iff $M\cong \Hom_R(I,M)$.
\end{lemma} 

\begin{proof}
Considering the $\m$-adic completions, note that $\widehat I$ is an $\widehat \m$-primary ideal of $\widehat R$. Moreover, in view of \cite[Cor.~1.15]{LW}, the desired isomorphisms are equivalent to $\widehat I \ \widehat M \cong \widehat M$ and $\widehat M \cong \Hom_{\widehat R}(\widehat I, \widehat M)$ respectively. So, without loss of generality, we may assume that $R$ is complete, and hence it admits a canonical module, say $\omega$. Set $(-)^{\dagger} := \Hom_R(-, \omega)$.

($\Rightarrow$): In view of \cite[Thm.~4.6 and Lem.~4.15]{dms}, one obtains that $I M^{\dagger} \cong M^{\dagger}$. It is well known that $M\cong M^{\dagger \dagger}$. Hence, by Lemma~\ref{n1},
$$M\cong M^{\dagger \dagger}\cong \Hom_{R}\big(I M^{\dagger}, \omega\big)\cong \Hom_{R}(I, M^{\dagger \dagger}) \cong \Hom_{R}(I, M).$$

($\Leftarrow$): Since $M^{\dagger\dagger} \cong M $, one has that $M \cong \Hom_{R}( I,  M)\cong \Hom_{ R}( I, \Hom_{ R}( M^{\dagger}, \omega))\cong \Hom_{ R}( I  M^{\dagger}, \omega)$, where the last isomorphism follows from Lemma~\ref{n1}. Hence, taking canonical dual, $ M^{\dagger}\cong  I  M^{\dagger}$. Therefore, by \cite[Thm.~4.6 and Lem.~4.15]{dms}, $ I  M \cong  M$.
\end{proof} 

Now we prove the following.

\begin{proposition}  
Let $(R,\m)$ be a CM local ring of dimension $1$. Let $M$ be an MCM $R$-module, and $I$ be an $\m$-primary ideal of $R$. Suppose there exist positive integers $a$ and $b$ such that $I^aM\cong M$ and $\Ext^1_R(I^b, M)=0$. Then, $I\cong R$.  
\end{proposition}

\begin{proof}
By repeatedly applying \cite[Thm~4.6]{dms}, one obtains that $I^bM\cong M$. Hence, by Lemma~\ref{ne2}, $\Hom_R(I^b,M)\cong M$. Therefore, since $\Ext^1_R(I^b, M)=0$, in view of Proposition~\ref{prop:GT}, $ \mu_R(I^b)\mu_R^1(M) = \mu_R^1(\Hom_R(I^b,M)) = \mu_R^1(M) $, which implies that $\mu_R(I^b)=1$ since $ \mu_R^1(M) \neq 0 $. Hence $I^b\cong R$ (since $I^b$ contains a non-zero-divisor). This gives a surjection $I^{\oplus n}\to R\to 0$ (for some $n$), which splits, hence $R$ is a direct summand of $I^{\oplus n}$. As $R$ is local, by \cite[Lem.~1.2.(i)]{LW}, $R$ is a summand of $I$. So there exists an $R$-module $N$ and an $R$-linear isomorphism $f:R \oplus N\to I$. We claim that $N=0$. Let $x\in N$. Set $r:=f(0,x)$ and $s:=f(1,0)$. Then $r,s\in R$, and $f(0,sx)=sf(0,x)=sr=rs=rf(1,0)=f(r,0)$. Since $f$ is an isomorphism, it follows that $(0,sx)=(r,0)$, and hence $r=0$, i.e., $f(0,x)=0$, which implies that $x=0$. Thus $N=0$. So $I\cong R$.
\end{proof} 

We also use \Cref{ne2} to give a negative answer to a question of Gheibi-Takahashi on characteristic module as introduced in \cite{GT24}.

\begin{para}\label{para:char-mod}
    Let $(R,\m)$ be a local ring, and $M$ be an $R$-module. A characteristic module of $M$ is defined in \cite[Defn.~3.4]{GT24} to be an $R$-module $\T_M$ such that $\widehat{\T_M} \cong \Tor_{\dim(Q)-\depth(R)}^Q(\widehat{R},\widehat{M})$, where $Q \twoheadrightarrow\widehat{R}$ is a Cohen presentation of $R$, and $\widehat{(-)}$ denotes the $\m$-adic completion. Suppose that $R$ is complete. Then $\T_M$ exists, and $\T_M\cong \mathbb{T}(M)$ in the sense of \cite[Defn.~4.1]{GT24}. Thus, in view of \cite[Prop.~4.2]{GT24}, $\mathbb{T}(M)\cong \Hom_R(E,M)$, where $E$ is a quasi-canonical module of $R$ (according to \cite[Defn.~3.1 and Prop.~3.2.(1)]{GT24}). Suppose that $R$ is CM. Then $R$ admits a canonical module $\omega$, and $E\cong \omega$ (by \cite[Prop.~3.2.(3)]{GT24}). It follows that $\T_M\cong \mathbb{T}(M) \cong \Hom_R(\omega,M)$.
\end{para}

\begin{question}(\cite[Ques.~5.4]{GT24})\label{ques:GT}
Let $R$ be a local ring. Suppose there is an $R$-module $M$ such that $\dim(M)=\dim(R)$, its characteristic module $\T_M$ exists, and $\T_M\cong M$. Then is $R$ Gorenstein?
\end{question}

\begin{para}\label{para:blow-up}
    For an $\m$-primary ideal $I$ of a CM local ring $(R,\m)$, the blow-up $B(I)$ of $I$ is defined as $\bigcup_{n\ge 0}(I^n:_{Q(R)} I^n)$, where $Q(R)$ denotes the total ring of fractions of $R$. It is clear that $R\subseteq B(I)\subseteq Q(R)$ and well known that  $B(I)$ is a module finite ring extension of $R$. Consequently, $B(I)$ is an MCM $R$-module, and also it is CM as a ring. For more details, we refer the reader to \cite[Sec.~4]{dms} and \cite[Sec.~2]{BP95}.  
\end{para}  

\begin{para}\label{para:GT-ans}
    The answer to \Cref{ques:GT} is negative as shown in \Cref{thm:negative-answer-GT}. Note that any numerical semigroup ring $R =k[[S]]$ satisfies all the hypotheses of the following theorem, and there exists non-Gorenstein numerical semigroup rings, for examples, $R = k[[t^n,t^{n+1},\dots,t^{2n-1}]]$, where $n\ge 3$.
\end{para}

\begin{theorem}\label{thm:negative-answer-GT}
Let $(R,\m)$ be a CM complete local ring of dimension $1$. Suppose that $R$ is generically Gorenstein. Then, $R$ has a canonical ideal $\omega$.
\begin{enumerate}[\rm (1)]
    \item For every MCM $R$-module $M$, $\T_M$ exists, and moreover $\T_M\cong M$ if and only if $M\cong \omega M$.
    \item In particular, the blow-up $B(\omega)$, which is an MCM $R$-module, satisfies $\T_{B(\omega)}\cong B(\omega)$, hence satisfies all the hypotheses of {\rm \Cref{ques:GT}}.
\end{enumerate}
\end{theorem} 

\begin{proof}
Since $R$ is CM and complete local, it admits a canonical module $\omega$. This can be identified with an $\m$-primary ideal as $R$ is generically Gorenstein of dimension $1$, see, e.g., \cite[3.3.18]{BH93}.

(1) Let $M$ be an MCM $R$-module. In view of \ref{para:char-mod}, one has that $\T_M\cong \Hom_R(\omega, M)$. Hence it follows from \Cref{ne2} that $\T_M\cong M$ if and only if $M\cong \omega M$.

(2) In view of \Cref{para:blow-up}, by \cite[Thm.~4.6.$(1)\Leftrightarrow(6)$]{dms}, $B(\omega)\cong \omega B(\omega)$. Therefore, by (1), $B(\omega) \cong \T_{B(\omega)}$. 
\end{proof}

We conclude this section by recording some consequences of finite projective dimension of $M^*$ under some Ext vanishing hypotheses. The first lemma towards this shows that every module whose Auslander transpose has projective dimension at most $1$ must be stably isomorphic to its torsion submodule.





\begin{lemma}\label{trt}
    Let $M$ be an $R$-module such that $\pd_R(\Tr M)\le 1$. Then, $M\cong t(M)\oplus G$ for some projective $R$-module $G$. 
\end{lemma}

\begin{proof}
We first do the case when $M$ is torsion-free, i.e., $t(M)=0$.  Since $\pd_R(\Tr M)\le 1$, one obtains that $\Ext^2_R(\Tr M, R)=0$. Since $\pd_R(\Tr M)\le 1$, by Auslander-Buchsbaum formula, $\Tr M$ is locally free on $\Ass(R)$, hence $\Ext^1_R(\Tr M, R)$ vanishes locally on $\Ass(R)$, i.e., $\Ext^1_R(\Tr M, R)$ is a torsion module. Remember that $\Ext^1_R(\Tr M, R)$ embeds in $M$, see, e.g., \cite[1.4.21]{BH93}. Therefore, since $M$ is torsion-free, $\Ext^1_R(\Tr M, R)$ has to be torsion-free as well. Thus $\Ext^1_R(\Tr M, R)$ must be zero. So $M$ is reflexive (cf.~\cite[1.4.21]{BH93}). Since $\pd_R(\Tr M)\le 1$, in view of the exact sequence \eqref{eqn:es-M*-Tr-M}, it follows that $M^*$ is projective. Hence $M\cong M^{**}$ is also projective. 

Now we tackle the general case. Consider the short exact sequence $0\to t(M)\to M\to M/t(M)\to 0$, which gives rise to an exact sequence (see \cite[Lemma 6]{mas})
$$0\to (M/t(M))^*\to M^* \to (t(M))^*\to \Tr(M/t(M))\to \Tr M\to \Tr(t(M))\to 0.$$
Since $t(M)$ is a torsion module, it follows that $(t(M))^*=0$, which yields a short exact sequence $0\to \Tr(M/t(M))\to \Tr M\to \Tr(t(M))\to 0$. Also $(t(M))^*=0$ implies that $\pd_R(\Tr(t(M)))\le 1$.
Therefore, since $\pd_R(\Tr M)\le 1$, one obtains that $\pd_R(\Tr(M/t(M))\le 1$.
Since $M/t(M)$ is torsion-free, from the torsion-free portion proved earlier, we deduce that $M/t(M)$ is projective, and so the short exact sequence $0\to t(M)\to M\to M/t(M)\to 0$ splits, proving our claim.   
\end{proof}

Using Lemma~\ref{trt}, we prove that every module $M$ satisfying $\Ext_R^{1\le i \le \depth(R)-1}(M,R)=0$ and $\pd_R(M^*)<\infty$ must be stably isomorphic to its torsion submodule.  


\begin{proposition}\label{trt2}
Let $R$ be a local ring of depth $t$, and let $M$ be an $R$-module such that $\pd_R(M^*)<\infty$.
\begin{enumerate}[\rm (1)]
    \item If $\Ext_R^{1\le i \le t-1}(M,R)=0$, then $M\cong t(M)\oplus G$ for some free $R$-module $G$.
    \item If $\Ext_R^{1\le i \le t}(M,R)=0$, then $M$ is free.
\end{enumerate} 
\end{proposition} 

\begin{proof}
Let $\Ext_R^{1\le i \le n}(M,R)=0$ for some $n$. As $\pd_R(M^*)<\infty$, it follows that $\pd_R(\Tr M)<\infty$. Since $M$ is stably isomorphic to $\Tr(\Tr M)$, we have that $\Ext_R^{1\le i\le n}\left(\Tr(\Tr M),R\right)=0$. Thus $\Tr M$ is $n$-torsion-free. Hence, by \cite[Prop.~11.(c)]{mas}, $\depth(\Tr M) \ge \inf\{n, t \}$.

(1) Since $\Ext_R^{1\le i \le t-1}(M,R)=0$, from the above discussion, $\depth(\Tr M) \ge \inf\{t-1, t \} = t-1$. So, by Auslander-Buchsbaum formula, $\pd_R(\Tr M) = t-\depth(\Tr M) \le 1$. Therefore, by Lemma~\ref{trt}, $M\cong t(M)\oplus G$ for some free $R$-module $G$.

(2) Since $\Ext_R^{1\le i \le t}(M,R)=0$, in this case $\depth(\Tr M) \ge \inf\{t, t \} = t$, and hence $\Tr M$ is free by Auslander-Buchsbaum formula. Being stably isomorphic to $\Tr(\Tr M)$, the module $M$ is also free. 
\end{proof}

\section{Finite injective dimension of Hom and vanishing of Ext}\label{sec:id-hom}

Here we study the consequences of $\Hom_R(M,N)$ having finite injective dimension under the vanishing condition of certain $\Ext_R^i(M,N)$. We start with the following theorem, which generalizes \cite[Cor.~2.10.(2)]{GT21}. Indeed, taking $N=R$ in the theorem below, one recovers \cite[Cor.~2.10.(2)]{GT21}.

\begin{theorem}\label{inj1}
Let $M$ and $N$ be nonzero modules over a local ring $R$ such that
$$ \Ext_R^{1\le i\le \depth(R)}(M,R) = 0 = \Ext_R^{1\le j\le \depth(N)}(M,N), $$
and $\id_R(\Hom_R(M,N))<\infty$. Then, $M$ is free, and consequently, $\id_R(N)<\infty$.  
\end{theorem}

\begin{proof}
Set $t=\depth(R)$. Since $M$ is stably isomorphic to $\Tr(\Tr(M))$, by the given hypothesis, $\Ext_R^{1\le i\le t}\left(\Tr(\Tr M),R\right)=0$. Thus $\Tr(M)$ is $t$-torsion-free. Therefore, by \cite[Prop.~11.(c)]{mas}, $\depth(\Tr M) \ge \inf\{t, \depth R \}=t$. Hence, since $\id_R(\Hom_R(M,N))<\infty$, it follows that 
\begin{equation}\label{Ext-of-Tr-vanishing}
    \Ext^{>0}_R\left(\Tr M, \Hom_R(M,N)\right) = 0,
\end{equation}
see, e.g., \cite[3.1.24]{BH93}. So the claim of $M$ being free follows from Proposition~\ref{propnew}. Consequently, since $\id_R(\Hom_R(M,N))<\infty$, one has that $\id_R(N)<\infty$.
\end{proof}

Next we considerably strengthen \cite[Thm.~2.15]{GT21}.
Indeed, taking even $N=M\neq 0$ in the following theorem, one recovers and improves \cite[Thm.~2.15]{GT21}. Here we need first $\depth(R)$ number of vanishing of $\Ext_R^i(M,R)$, unlike in \cite[Thm.~2.15]{GT21}, where $(2\depth(R)+1)$ number of vanishing of Ext is required.

\begin{theorem}\label{injimprov}
Let $M$ and $N$ be modules over a local ring $R$. Set $t:=\depth(R)$. Suppose that $\Hom_R(M,N)$ is nonzero, and $\Ext_R^i(M,R)=\Ext_R^j(M,N)=0$ for all $1\le i\le t$ and $1\le j\le t -1$. Assume that $\id_R(\Hom_R(M,N))$ is finite. Then, $M$ is free, and $\id_R(N) <\infty$.
\end{theorem}

\begin{proof}
Note that $R$ is CM by Bass’ Conjecture, see, e.g., \cite[9.6.2 and 9.6.4]{BH93}. As $R$ is CM and $N\neq 0$, we have that $\depth N\le t$. If $\depth N\le t-1$, then we are done by Theorem~\ref{inj1}. So we may assume that $\depth N=t$. Then, by \cite[Thm.~2.3]{GT21}, $\id_R(N) < \infty$, and $M\cong \Gamma_{\fm}(M)\oplus R^{\oplus r}$ for some $r\ge 0$. It remains to prove that $M$ is free. Two possible cases appear.

Case I:
Assume that $t=0$. Then $R$ is Artinian. Therefore both $N$ and $\Hom_R(M,N)$ are (nonzero) injective modules as they have finite injective dimension. Moreover $N\cong E^{\oplus n}$ and $\Hom_R(M,N) \cong E^{\oplus s}$ for some positive integers $n$ and $s$, where $E$ is the injective hull of the residue field of $R$. Hence
\begin{equation}\label{Hom-E-iso}
    E^{\oplus s} \cong \Hom_R(M,N) \cong \Hom_R(M,E)^{\oplus n}
\end{equation}
Dualizing \eqref{Hom-E-iso} with respect to $E$, and using Matlis duality, we get that $R^{\oplus s} \cong M^{\oplus n}$, which implies the direct summand $M$ is free. 

Case II: Suppose $t\ge 1$. If possible, assume that $\Gamma_{\fm}(M) \neq 0$. Since $\Gamma_{\fm}(M)$ is a module of finite length, and $t=\depth(R)$, it follows that $\Ext^t_R(\Gamma_{\fm}(M),R)\neq 0$. On the other hand, since $M\cong \Gamma_{\fm}(M)\oplus R^{\oplus r}$, from the given hypothesis, we obtain that $\Ext^{1\le i\le t}_R(\Gamma_{\fm}(M),R)=0$. Thus we get a contradiction. Therefore $\Gamma_{\fm}(M)=0$. It follows that $ M \cong R^{\oplus r} $, which is free.
\end{proof}

Theorems~\ref{inj1} and \ref{injimprov} provide partial positive answers to \cite[Ques.~2.9]{GT21}. Next we show that the vanishing condition on $\Ext_R^j(M,N)$ in Theorems~\ref{inj1} and \ref{injimprov} can be removed provided $\pd_R(N) < \infty$. For that, we need the following lemma.

\begin{lemma}\label{2.3}
Let $M$ and $N$ be $R$-modules over a local ring $R$. Let $h = \pd_R(N) < \infty$. Then the following statements hold true.
\begin{enumerate}[\rm(1)]
    \item If $\Ext_R^{1\le i\le h+1}(M,R)=0$, then $\Ext^1_R(M,N)=0$.
    \item If $\Ext^{1\le i\le s+h+1}_R(M,R)=0$ for some integer $s\ge 0$, then
    $$\Ext^{1\le i\le s+1}_R(M,N)=0.$$
\end{enumerate}  
\end{lemma}

\begin{proof}
(1) Note that the $h$-th syzygy module $\Omega^h(N)$ is free as $h = \pd_R(N)$. Since $\Ext_R^{1\le i\le h+1}(M,R)=0$, the exact sequences $0 \to \Omega^{j+1}(N) \to R^{\beta_j} \to \Omega^j(N) \to 0$ for all $0\le j\le h-1$ yield that
\begin{align*}
    \Ext_R^1(M,N) &\cong \Ext_R^2(M,\Omega^1(N)) \cong \Ext_R^3(M,\Omega^2(N)) \\
    &\cong \cdots \cong \Ext_R^h(M,\Omega^{h-1}(N)) \cong \Ext_R^{h+1}(M,\Omega^h(N)) = 0,
\end{align*}
where the last equality holds as $\Omega^h(N)$ is free.

(2) By hypothesis, $\Ext^{1\le i\le h+1}_R(\Omega^n_R(M),R)=0$ for all $0\le n \le s$.  Hence, by (1), $\Ext^1_R(\Omega^n_R(M),N)=0$ for all $0\le n\le s$. So $\Ext^{1\le i\le s+1}_R(M,N)=0$.  
\end{proof}

Now we see another generalization of \cite[Cor.~2.10.(2)]{GT21}. In fact, substituting $N=R$ in Corollary~\ref{injcor} below, one obtains \cite[Cor.~2.10.(2)]{GT21}.

\begin{corollary}\label{injcor}
Let $M$ and $N$ be nonzero modules over a local ring $R$ such that $\pd_R(N)<\infty$, 
$\Ext_R^{1\le i\le \depth(R)}(M,R)=0$ and $\id_R(\Hom_R(M,N))<\infty$. Then, $M$ is free, and $R$ is Gorenstein. 
\end{corollary}

\begin{proof}
If $\depth(N)=0$, then $M$ is free by Theorem~\ref{inj1}. So we may assume that $\depth(N)\ge 1$. Since $\pd_R(N)<\infty$, by Auslander-Buchsbaum formula, $\depth(R) = \depth(N) + \pd_R(N)$. Then, by the given hypothesis,
$$\Ext_R^{1\le i\le (\depth(N)-1)+\pd_R(N)+1}(M,R)=\Ext_R^{1\le i\le \depth(R)}(M,R)=0.$$
Therefore, in view of Lemma~\ref{2.3}(2), $ \Ext^{1\le i\le \depth(N)}_R(M,N) = 0$. Hence, by Theorem~\ref{inj1}, $M$ is free. 
As $M$ is free, we also get $ \id_R(N) < \infty $. Thus both $\pd_R(N)$  and $\id_R(N)$ are finite.  So, by Foxby's theorem, $R$ is Gorenstein.
\end{proof}


\section{Generalized residually faithful and vanishing of Ext}\label{sec:res-faithful}

The notion of residually faithful modules was introduced and studied in \cite[Sec.~5]{BV01} and later studied by Goto-Kumashiro-Loan \cite{GKL19}.

\begin{definition}\cite[Defn.~5.1]{BV01}, \cite[Defn.~3.1]{GKL19}\label{defn:res-faithful-GKL}
    Let $R$ be a CM local ring. An MCM $R$-module $M$ is said to be residually faithful if $M/\fq M$ is a faithful $R/\fq$-module for some parameter ideal $\fq$ of $R$.
\end{definition}

Taking cue from Definition~\ref{defn:res-faithful-GKL}, we introduce the following notions. 

\begin{definition}\label{defn:res-faithful}
Let $n\ge 0$. An $R$-module $M$ is called:
\begin{enumerate}[(1)]
    \item $n$-residually faithful (resp., strongly $n$-residually faithful) if there exists an $M$-regular (resp., $M$-regular as well as $R$-regular) sequence ${\bf x} := x_1,\dots,x_n$ in $R$ such that $M/{\bf x} M$ is a faithful $R/{\bf x} R$-module.
    \item 
    absolutely $n$-residually faithful if for every $M$-regular sequence ${\bf x} := x_1,\dots,x_n$, the module $M/{\bf x} M$ is faithful over $R/{\bf x} R$.
    \item
    universally $n$-residually faithful if for every sequence ${\bf x} := x_1,\dots,x_n$ regular on both $R$ and $M$, the module $M/{\bf x} M$ is faithful over $R/{\bf x} R$.    
    \item 
    As conventions, we take a (strongly/absolutely/universally) $0$-residually faithful $R$-module to be just a faithful $R$-module.
\end{enumerate}   
\end{definition} 

\begin{remark}\label{remrf}
    For an $R$-module $M$, consider the implications:
    \[
        \xymatrix{\mbox{absolutely $n$-residually faithful} \ar@{=>}[r]^{\rm (a)} \ar@{=>}[d]_{\rm (d)}^{\depth(M)\ge n} & \mbox{universally $n$-residually faithful} \ar@/{}^{1.3pc}/@{=>}[d]_{\rm (b)}^{\depth(M\oplus R)\ge n} \ar@/{}_{1.3pc}/@{<=>}[d]_{\rm MCM, dim=n} \\
        \mbox{$n$-residually faithful} & \mbox{strongly $n$-residually faithful} \ar@{=>}[l]_{\rm (c)}}
    \]
    \begin{enumerate}[(1)]
        \item 
        The implications (a) and (c) hold true always, while (d) is true if $\depth(M)\ge n$, and (b) is true if both $\depth(M)\ge n$ and $\depth(R)\ge n$. Note that if $\depth_R(M)<n$, then $M$ is vacuously  absolutely $n$-residually faithful, but not $n$-residually faithful. If $\depth M<n$ or $\depth R<n$, then vacuously $M$ is universally $n$-residually faithful, but not strongly $n$-residually faithful.
        \item 
        Let $M$ be an MCM module over a CM local ring $R$ of dimension $n$. Then, the converse to (b) is also true by \cite[Prop.~3.2]{GKL19}.
        \item By definition, $R$ is absolutely $n$-residually faithful $R$-module for every $n$.
    \end{enumerate}
\end{remark}

\begin{para}\label{para:R-S-res-faith}
    Let $(R,\m)\to (S,\fn)$ be a flat local homomorphism, where  $R$ is CM and $\fm S = \fn$. Then, an $R$-module $M$ is residually faithful \iff the $S$-module $S\otimes_R M$ is so. Indeed, since $R\to S$ is flat and $\fm S = \fn$, one has that $\dim_R(M) = \dim_S(S\otimes_R M)$ and $\depth_R(M) = \depth_S(S\otimes_R M)$. Therefore, since $R$ is CM, $S$ is also CM. Moreover, $M$ is MCM over $R$ \iff $S\otimes_R M$ is so over $S$. Hence the desired result follows from \cite[Prop.~3.6(2)]{GKL19}.
\end{para}

The following questions arise naturally from Remark~\ref{remrf}.

\begin{question}
    Regarding the implications in Remark~\ref{remrf}, for an MCM module over a CM local ring $R$ of dimension $>n$, is the converse to (b) still true? 
        
\end{question}


\begin{para}\label{para:surj-map-on-strong-n-resd}
    Let $M$ and $N$ be $R$-modules with a surjection $M^{\oplus s}\twoheadrightarrow N$ for some $s\ge 1$. Let $n\ge 0$ be an integer. If $M$ satisfies $(\widetilde S_n)$ and $N$ is strongly $n$-residually faithful, then $M$ is strongly $n$-residually faithful. Indeed, since $N$ is strongly $n$-residually faithful, by definition, there exists an $R$ and $N$-regular sequence ${\mathbf x}:=x_1,\dots,x_n$ such that $N/\mathbf x N$ is $R/\mathbf x R$-faithful. As $M$ satisfies $(\widetilde S_n)$, the sequence ${\bf x}$ is also $M$-regular (\cite[Prop.~2]{Mal68}). Tensoring $M^{\oplus s}\twoheadrightarrow N$ with $R/\mathbf x R$, we get another surjection $(M/\mathbf x M)^{\oplus s} \twoheadrightarrow N/\mathbf x N $. Hence the claim follows.
\end{para}

A $($strongly$)$ $n$-residually faithful module is $($strongly$)$ $(n-1)$-residually faithful.

\begin{lemma}
Let $M$ be a $($strongly$)$ $n$-residually faithful module over a local ring $R$ for some $n\ge 1$. Then, $M$ is $($strongly$)$ $(n-1)$-residually faithful. In particular, any $n$-residually faithful module is faithful.
\end{lemma}

\begin{proof}
First we consider the case $n=1$. Then there exists an ($R$-regular as well as) $M$-regular element $x$ in $R$ such that $M/xM$ is $R/xR$-faithful. We claim that $M/x^tM$ is $R/x^tR$-faithful for every $t\ge 1$. In order to show this, let $r\in R$ be such that $r+x^tR$ annihilates $M/x^tM$. Then $rM\subseteq x^t M$. In particular, $rM\subseteq xM$, i.e., $r+xR\in \ann_{R/xR}(M/xM)$. So $r\in xR$, and hence $r=xr_1$ for some $r_1\in R$. Now $xr_1M\subseteq x^t M$ and $x$ being $M$-regular imply that $r_1 M\subseteq x^{t-1}M$. If $t\ge 2$, repeating the same procedure, inductively, we end up with $r=x^tr_t\in x^tR$, which shows that $M/x^tM$ is $R/x^tR$-faithful for every $t\ge 1$. Thus $\ann_R(M)\subseteq \bigcap_{t\ge 1}\ann_R(M/x^t M) =\bigcap_{t\ge 1} x^tR=0$, where the last equality follows from Krull's intersection theorem. Therefore $M$ is faithful over $R$.

Let $n \ge 2$. There exists an ($R$-regular as well as) $M$-regular sequence ${\bf x} := x_1,\dots,x_n$ in $R$ such that $M/{\bf x} M$ is $R/{\bf x} R$-faithful. Then $M/(x_1,\dots,x_{n-1})M$ is (strongly) $1$-residually faithful over $R/(x_1,\dots,x_{n-1})R$, hence $M/(x_1,\dots,x_{n-1})M$ is faithful over $R/(x_1,\dots,x_{n-1})R$. Thus, by definition, $M$ is (strongly) $(n-1)$-residually faithful.  
\end{proof}

For a module $L$ over a local ring $R$ of depth $t$, the syzygy module $\syz_R^{t+1}(L)$ is never strongly $t$-residually faithful.  


\begin{lemma}\label{syzf}
Let $(R,\m)$ be a local ring of depth $t$. Then, the following hold true:
\begin{enumerate}[\rm(1)]
    \item Let $M$ be an $R$-module satisfying $(\widetilde S_t)$ $($cf.~{\rm\ref{para-Sn}}$)$. Then $\syz_R(M)$ is not strongly $t$-residually faithful. In particular, for an $R$-module $L$, $\syz_R^{t+1}(L)$ is not strongly $t$-residually faithful.
    \item Assume that $\depth(M) \ge t$. Then, $\syz_R(M)$ is not universally $t$-residually faithful. In particular, for an $R$-module $L$, $\syz_R^{t+1}(L)$ is not universally $t$-residually faithful. 
\end{enumerate}
\end{lemma}   

\begin{proof}
Set $N:=\syz_R(M)$. 

(1) If possible, let $N$ be strongly $t$-residually faithful. Then there exists a sequence $ {\bf x} = x_1,\dots,x_t $ regular on $R$ and $N$ such that $N/{\bf x} N$ is $R/{\bf x} R$-faithful. As $M$ satisfies $(\widetilde S_t)$, the sequence ${\bf x}$ is also $M$-regular (\cite[Prop.~2]{Mal68}). Hence $\Tor^R_1(M,R/{\bf x} R)=0$. We have a short exact sequence $ 0 \to N \to R^{\oplus \mu(M)} \to M \to 0 $, tensoring which by $R/{\bf x} R$, we get another short exact sequence $0\to N/{\bf x} N \xrightarrow{f} (R/{\bf x} R)^{\oplus \mu(M)}\to M/{\bf x}M \to 0$. Therefore, since $\mu_R(M)=\mu_{R/{\bf x} R}(M/{\bf x} M)$, it follows that $\text{Image}(f)\subseteq \m (R/{\bf x} R)^{\oplus \mu(M)}$, hence $N/{\bf x} N \cong \text{Image}(f)$ is annihilated by $\Soc(R/{\bf x} R)$, which is nonzero since $\depth(R/{\bf x} R)=0$. This is a contradiction. Therefore $N$ is not strongly $t$-residually  faithful.

The last part of the claim about $\syz_R^{t+1}(L)$ follows since $\syz^t_R(L)$ satisfies $(\widetilde S_t)$.

(2) Since $\depth(M)\ge t$, we have that $\depth(N)\ge t$. Choose a sequence ${\bf x} := x_1,\dots,x_t$ that is regular on $R,M$ and $N$. Hence $\Tor^R_1(M,R/{\bf x} R)=0$. We have a short exact sequence $ 0 \to N \to R^{\oplus \mu(M)} \to M \to 0 $, tensoring which by $R/{\bf x} R$, we get another short exact sequence $0\to N/{\bf x} N \xrightarrow{f} (R/{\bf x} R)^{\oplus \mu(M)}\to M/{\bf x}M \to 0$. Therefore, since $\mu_R(M)=\mu_{R/{\bf x} R}(M/{\bf x} M)$, it follows that $\text{Image}(f)\subseteq \m (R/{\bf x} R)^{\oplus \mu(M)}$, hence $N/{\bf x} N \cong \text{Image}(f)$ is annihilated by $\Soc(R/{\bf x} R)$, which is nonzero since $\depth(R/{\bf x} R)=0$. Therefore, $N$ is not universally $t$-residually faithful by definition.
\end{proof} 

In view of \cite[Cor.~3.4]{GKL19} and Remark \ref{remrf}, Proposition~\ref{resfaith2} highly generalizes \cite[Thm.~3.9]{GKL19} with even fewer vanishing conditions. 

\begin{proposition}\label{resfaith2}
Let $n\ge 0$, and $M,N$ be $R$-modules with $\Ext_R^{1\le i \le n-1}(M,N)=0$. If $\Hom_R(M,N)$ is absolutely $($resp., universally$)$ $n$-residually faithful, then so is $N$.
\end{proposition}

\begin{proof}
    For the $n=0$ case, it is clear that if $\Hom_R(M,N)$ is faithful, then so is $N$, since $\Hom_R(M,N)$ embeds inside a finite direct sum of copies of $N$. So we assume that $n\ge 1$.  Let ${\bf x} := x_1,\dots,x_n$ be elements of $R$ which is $N$-regular $($resp., and $R$-regular$)$. We want to show that $N/{\bf x} N$ is $R/{\bf x} R$-faithful. In view of Lemma~\ref{extres}, the sequence ${\bf x}$ is also $\Hom_R(M,N)$-regular, and $\Hom_R(M,N)/{\bf x} \Hom_R(M,N)$ embeds in $\Hom_R(M,N/{\bf x} N)$. Thus, since $\Hom_R(M,N)$ is absolutely $n$-residually faithful, $\Hom_R(M,N)/{\bf x}\Hom_R(M,N)$ is $R/{\bf x} R$-faithful. Combining these, it follows that
    \begin{align*}
        \ann_R(N/{\bf x} N) &\subseteq \ann_R \Hom_R(M,N/{\bf x} N)\\
        &\subseteq \ann_R \left( \Hom_R(M,N)/{\bf x} \Hom_R(M,N) \right) = {\bf x} R.
    \end{align*}
    Thus, $N/{\bf x} N$ is $R/{\bf x} R$-faithful, which is what we wanted to show. 
\end{proof}



Every $n$-semidualizing module is absolutely $(n+1)$-residually faithful.

\begin{proposition}\label{semid}
    Let $n\ge 1$, and $M$ be an $(n-1)$-semidualizing $R$-module $($see {\rm Definition~\ref{defn:n-semidualizing}}$)$. Then, $M$ is absolutely $n$-residually faithful.
\end{proposition}

\begin{proof}
    Since $M$ is $(n-1)$-semidualizing, by definition,  $\Hom_R(M,M) \cong R$ and $\Ext_R^{1\le i \le n-1}(M,M) = 0$. So the claim follows from Proposition~\ref{resfaith2}. 
\end{proof}

\begin{proposition}
Let $R$ be a local ring of depth $t$. Let $M$ be an $R$-module satisfying $(\widetilde S_t)$, and $X$ be a $(t-1)$-semidualizing $R$-module. Then, there is no surjection from a direct sum of copies of $\syz(M)$ onto $X$. 
\end{proposition}

\begin{proof}
Since $X$ is $(t-1)$-semidualizing, in particular $\Hom_R(X,X)\cong R$, and hence $X$ is $R$-faithful, i.e., absolutely $0$-residually faithful. Moreover, in view of Proposition~\ref{semid}, $X$ is absolutely $t$-residually faithful. Since $X$ is $(t-1)$-semidualizing and $\depth(R) = t$, by \cite[Prop.~3.2.(a)]{Se}, $\depth(X) \ge t$. So $X$ is strongly $t$-residually faithful. As $\syz(M)$ is not strongly $t$-residually faithful by Lemma~\ref{syzf}.(1), the non-existence of any surjection as mentioned in the proposition follows from \ref{para:surj-map-on-strong-n-resd}.
\end{proof}


Next, we provide some criteria, in terms of the notion of residually faithful modules (Definition~\ref{defn:res-faithful-GKL}), for an $\fm$-primary ideal in a CM local ring $(R,\m)$ to be generated by a regular sequence. For this, we first record a lemma.

\begin{lemma}\label{lem:J=x}
    Let $(R,\m)$ be a CM local ring, $M$ be an MCM $R$-module, and $J$ be an $\m$-primary ideal of $R$. Furthermore, assume that $e_R(J,M)=\ell_R(M/JM)$ and $M$ is residually faithful. Then $J$ can be generated by an $R$-regular sequence.
\end{lemma}

\begin{proof}  
    If necessary, in view of Lemma~\ref{lem:JS-J-reg-gen} and  \ref{para:R-S-res-faith}, 
    $R$ can be replaced by $R[X]_{\fm R[X]}$. So we may assume that the residue field of $R$ is infinite. Since $e_R(J, M)=\ell_R(M/JM)$, by Lemma~\ref{lem:JM=xM}, $JM=(\mathbf x )M$ for some $R$-regular sequence ${\mathbf x} = x_1,\ldots,x_d$, where $d=\dim(R)$. Since $M$ is MCM, ${\mathbf x}$ is also regular on $M$. On the other hand, since $R$ is CM, every parameter ideal of $R$ is generated by a maximal $R$-regular (hence $M$-regular) sequence. Thus, since $M$ is residually faithful, by definition, $M$ is strongly $d$-residually faithful. Hence Remark~\ref{remrf}(2) yields that $M$ is universally $d$-residually faithful. So $M/{\bf x}M$ is faithful over $R/{\bf x}R$. Since $JM=({\bf x})M$, the module $M/{\bf x}M$ is killed by $J$. Putting together, $J=({\bf x})$.
\end{proof}

    

Next we show a couple of lemmas in order to prove Proposition~\ref{resfaithul}. For both the lemmas, we need the notion of trace ideals (cf.~\ref{para:trace}).

\begin{lemma}\label{tracereg}
Let $M$ be an $R$-module such that $\Ext^1_R(M,R)=0$. Let $x\in R$ be $R$-regular. Set $ \overline{(-)} := (-) \otimes_R R/(x) $. Then, the following hold true.
\begin{enumerate}[\rm (1)]
    \item  $(\tr_R M)\overline R=\tr_{\overline R}(\overline M)$. 
    \item If $(R,\m)$ is local, $x\in \m$, and $\overline R$ is a direct summand of $\overline M$, then $R$ is a direct summand of $M$. 
\end{enumerate}
\end{lemma}

\begin{proof}
(1) By Lemma~\ref{extres}(1)(iii), we have a natural isomorphism $\overline{\Hom_R(M,R)}\cong \Hom_{\overline R}(\overline M, \overline R)$. Now the claim follows along the same lines as argued in \cite[Lem.~1.5(iii)]{hhs}.  

(2) Since $\overline R$ is a direct summand of $\overline M$, by \cite[Prop.~2.8(iii)]{Lindo}, $\tr_{\overline R}(\overline M)=\overline R$. Hence (1) yields that $(\tr_R M)\overline R=\overline R$, i.e., $\tr_R(M)+Rx=R$. Since $x\in \fm$, it follows that $\tr_R(M)=R$. Therefore, again by \cite[Prop.~2.8(iii)]{Lindo}, $R$ is a direct summand of $M$. 
\end{proof}

The following lemma, whose proof does not depend on any result of \cite{DEL21}, highly generalizes \cite[Lem.~3.2]{DEL21}.
In fact, the case $t=0$ of Lemma~\ref{freesumm} recovers \cite[Lem.~3.2]{DEL21} with a new proof.

\begin{lemma}\label{freesumm}
Let $R$ be a local ring of depth  $t$. Let $M$ be an $R$-module such that $\depth_R(M)\ge t-1$ and $\Ext_R^{1\le i\le t-1}(M,R)=0$. If $R$ is a direct summand of $M^*$, then $R$ is a direct summand of $M$.
\end{lemma}    

\begin{proof}
Let $R$ be a direct summand of $M^*$. Using induction on $t$, we prove that $R$ is a direct summand of $M$. We first consider the case when $t\le 1$. (Note that in this case, there is no $\Ext$ vanishing). We first show that the trace ideal $\tr_R(M)$ contains a non-zero-divisor. If not, then by prime avoidance, $\tr_R(M)\subseteq \p=\ann_R(x)$ for some $\p\in \Ass(R)$, and $0\neq x \in R$. Consequently, $x$ annihilates $\tr_R(M)$, and hence $x$ annihilates $\Hom_R(M, \tr_R(M)) \cong M^*$, contradicting $R$ is a direct suumand of $M^*$. Thus, $\tr_R(M)$ must have positive grade. Then, by \cite[Coro.~3.23]{Lindo}, $\End_R(\tr_R(M))\cong \End_R(\tr_R(M^*))$. Since $R$ is a direct summand of $M^*$, in view of \cite[Prop.~2.8(iii)]{Lindo}, $\tr_R(M^*)=R$. Combining the last two facts, $\End_R(\tr_R(M))\cong R$. Hence, by \cite[Prop.~2.8(vi)]{Lindo}, $R\cong (\tr_R(M))^*$. As $\tr_R(M)$ is an ideal, it is torsion-free. Therefore, since $t\le 1$, in view of Proposition~\ref{trt2}.(1), $\tr_R(M)$ must be free. As $\tr_R(M)$ is an ideal, it follows that $\tr_R(M)=R$. Then \cite[Prop.~2.8(iii)]{Lindo} yields that $R$ is a direct summand of $M$.

Next we assume that $t\ge 2$. Note that $\depth_R(R\oplus M) \ge t-1\ge 1$. So there exists an element $x$ of $R$ that is regular on both $R$ and $M$. Set $\overline{(-)} := (-)\otimes_R R/(x)$. Since $\Ext_R^{1\le i\le t-1}(M,R)=0$, one obtains that $\Ext_{\overline{R}}^{1\le i\le t-2}(\overline{M},\overline{R})=0$, see, e.g., \cite[Lem.~2.5]{Gh19}. Moreover, since $\Ext_R^1(M,R)=0$, by Lemma~\ref{extres}(1)(iii), it follows that $\overline{(M^*)} \cong \Hom_{\overline{R}}(\overline{M},\overline{R})$. Note that $\depth(\overline{R}) = t-1$, $\depth_{\overline{R}}(\overline{M}) = \depth_R(M)-1\ge t-2$, and $\overline{R}$ is a direct summand of $ \Hom_{\overline{R}}(\overline{M},\overline{R})$. Therefore, by the induction hypothesis, $\overline{R}$ is a direct summand of $\overline{M}$. Finally, by Lemma~\ref{tracereg}(2), $R$ is a direct summand of $M$.
\end{proof}

Now we are able to give a criterion for an Ulrich module with respect to an $\fm$-primary ideal to be free using the notion of residually faithful modules. 

\begin{proposition}\label{resfaithul}
Let $(R,\m)$ be a CM local ring of dimension $d$. Let $J$ be an $\m$-primary ideal. Assume that there exist $R$-modules $M$ and $N$ such that $N$ is MCM, $\Ext_R^{1\le i\le d-1}(M,N)=0$, and $\Hom_R(M,N)$ is residually faithful.  Then, $J$ is generated by an $R$-regular sequence if at least one of the following conditions holds:
\begin{enumerate}[\rm (1)]
    \item $e_R(J,N)=\ell_R(N/JN)$.
    \item $N$ is Ulrich with respect to $J$ $(${\rm cf.~Definition~\ref{defn:Ulrich}}$)$.
\end{enumerate}
Moreover, under the hypothesis of $(2)$, $N$ must be free. Furthermore, if both $(2)$ and $\depth_R(M)\ge d-1$ hold, then $R$ is a direct summand of $M$.
\end{proposition}

\begin{proof} 
Note that the hypothesis of (2) contains the hypothesis of (1). Under the hypothesis of (1), by Proposition~\ref{mainul}, we have that $e_R\big(J,\Hom_R(M,N)\big) = \ell_R\big(\Hom_R(M,N)/J\Hom_R(M,N)\big)$. Hence, since $\Hom_R(M,N)$ is residually faithful, Lemma~\ref{lem:J=x} yields that $J$ is generated by an $R$-regular sequence. Under the hypothesis of (2), we additionally have that $N/JN$ is free over $R/J$. In this case, since $J$ is generated by an $R$-regular sequence and $N$ is MCM, it follows that $N$ must be free. Now assume, along with the hypothesis of (2), that $\depth_R(M)\ge d-1$. As $\Hom_R(M,N)$ is residually faithful, $N$ must be a nonzero module. Thus, since $N$ is free, $\Hom_R(M,N) \cong (M^*)^{\oplus n}$ for some $n \ge 1$ and $\Ext_R^{1\le i\le d-1}(M,R)=0$. Moreover, $(M^*)^{\oplus n}$ is residually faithful, i.e., $M^*$ is so. Since $\Ext_R^{1\le i\le d-1}(M,R)=0$, in view of \ref{para:Ext-syz}, it follows that $M^*\cong F\oplus \syz_R^{d+1}(L)$ for some $R$-modules $F$ and $L$ with $F$ free. By Lemma~\ref{syzf}(1), $\syz^{d+1}_R(L)$ is not residually faithful. Hence $F \neq 0$. So $R$ is a direct summand of $M^*$. Finally, by Lemma~\ref{freesumm}, we get that $R$ is a direct summand of $M$.
\end{proof}

As consequences of Proposition~\ref{resfaithul}, we provide characterizations of regular local rings via existence of certain residually faithful module.

\begin{corollary}\label{cor:rlr-resfaith}
    For a CM local ring $(R,\fm)$ of dimension $d$, the following are equivalent.
    \begin{enumerate}[\rm (1)]
        \item $R$ is regular.
        \item $R$ admits an Ulrich residually faithful module.
        \item There exist $R$-modules $M$ and $N$ such that $N$ is Ulrich, $\Ext_R^{1\le i\le d-1}(M,N)=0$, and $\Hom_R(M,N)$ is residually faithful.
    \end{enumerate}
\end{corollary}

\begin{proof}
    (1) $\Rightarrow$ (2): If $R$ is regular, then $M:=R$ is Ulrich and residually faithful as an $R$-module.

    (2) $\Rightarrow$ (3): Suppose that $R$ admits an Ulrich residually faithful module, say $N$. Set $M:=R$. Then the condition (3) is satisfied trivially.

    (3) $\Rightarrow$ (1). Since $N$ is Ulrich, $e_R(\fm,N)=\mu(N) = \ell_R(N/\fm N)$. Therefore, by Proposition~\ref{resfaithul}, the condition (3) implies that $\fm$ is generated by an $R$-regular sequence, i.e., $R$ is regular
\end{proof}

We need the following notion to prove our main result of this section.

\begin{para}\label{para:TD}
    Let $C$ be a semidualizing $R$-module. Following \cite[1.8]{RW}, the Bass class $\mathscr{B}_C$ with respect to $C$ is the collection of all $R$-modules $M$ satisfying
    \begin{enumerate}[\rm (a)]
        \item $\Ext_R^{\ge 1}(C,M)=0=\Tor_{\ge 1}^R(C,\Hom_R(C,M))$, and
        \item The natural map $C \otimes_R \Hom_R(C,M) \longrightarrow M$ is an isomorphism.
    \end{enumerate}
    Consider an $R$-module $X$. If $ \pd_R(\Hom_R(C,X)) < \infty $, then by \cite[1.9.(b) and Thm.~2.8.(a)]{RW}, it follows that $X\in \mathscr{B}_C$, consequently $\Ext_R^{\ge 1}(C, X)=0$ and $C\otimes_R\Hom_R(C,X)\cong X$.
\end{para}

Using the notion of generalized residually faithful modules, we prove the following theorem, which highly generalizes \cite[Cor.~3.7]{DEL21}.

\begin{theorem}\label{thm:pd-Hom-finite-gdim-freeness-criteria}
    Let $R$ be a local ring of depth $t$. Let $C$ be a semidualizing $R$-module $($e.g., $C=R$$)$. Let $M$ and $N$ be $R$-modules such that $\Hom_R(M,N)$ is nonzero,
    \[
        \Ext_R^{1\le i \le t-1}(M,N) = 0 \;\mbox{ and } \;\pd_R\big(\Hom_R(C,\Hom_R(M,N))\big) < \infty.
    \]
    Let $N \cong F \oplus \syz_R(Y)$ for some $R$-modules $Y$ and $F$ each of depth at least $t$ $($e.g., if $N$ is $(t+1)$-torsion-free, or if $\gdim_R(N)=0$$)$. Then, $N\cong F$, $C\cong R$ and $M\cong t(M)\oplus G$ for some nonzero free $R$-module $G$.
\end{theorem}

\begin{proof}
Without loss of generality, we may assume that $R$ is complete. Since $\depth(Y)\ge t$ and $\depth(F)\ge t$, it follows from $ N \cong F \oplus \syz_R(Y) $ that $\depth(N) \ge t$. Thus, since $\Ext_R^{1\le i \le t-1}(M,N)=0$, by Lemma~\ref{extres}.(2), $\depth(\Hom_R(M,N)) \ge \inf\{t, \depth(N)\} = t.$
Since $C$ is semidualizing and $\Hom_R(C,\Hom_R(M,N))$ has finite projective dimension, in view of \ref{para:TD}, one has that $\Ext_R^{\ge 1}(C,\Hom_R(M,N))=0$. Hence, applying Lemma~\ref{extres}.(2) again, it follows that
\begin{equation*}
    \depth\big(\Hom_R(C, \Hom_R(M,N))\big) \ge \inf\{t, \depth(\Hom_R(M,N))\} = t.
\end{equation*}   
Then, by Auslander-Buchsbaum formula, $\Hom_R(C,\Hom_R(M,N))$ is $R$-free. Therefore, by \ref{para:TD},
\begin{equation}\label{Hom-tensor-C}
    \Hom_R(M,N) \cong C \otimes_R \Hom_R(C,\Hom_R(M,N)) \cong C^{\oplus n}
\end{equation}
for some $n \ge 0$. Since $\Hom_R(M,N)\neq 0$, we must have that $n\ge 1$. As $C$ is semidualizing, $\Supp(C) = \Spec(R)$. So, from 
\eqref{Hom-tensor-C}, one derives that $\Supp(M) = \Spec(R)$. If possible, assume that $\syz_R(Y) \neq 0$. Then $\emptyset \neq \Ass (\syz_R(Y)) = \Ass(\syz_R(Y)) \cap \Spec(R) = \Ass (\syz_R(Y)) \cap \Supp(M) = \Ass\big(\Hom_R(M,\syz_R(Y))\big)$ (cf.~\cite[1.2.27]{BH93}). It follows that $\Hom_R(M,\syz_R(Y)) \neq 0$. Note that $\Hom_R(M,\syz_R(Y))$ is a direct summand of $\Hom_R(M,N) \cong C^{\oplus n}$. Being a semidualizing module, $C$ is indecomposable. Combining the last three facts, since $R$ is complete, by Krull-Schmidt theorem, one obtains that $\Hom_R(M,\syz_R(Y)) \cong C^{\oplus s}$ for some $s\ge 1$. Thus, by Proposition~\ref{semid}, $\Hom_R(M,\syz_R(Y))$ is absolutely $t$-residually faithful. So Proposition~\ref{resfaith2} yields that $\syz_R(Y)$ is absolutely $t$-residually faithful, and hence universally $t$-residually faithful. This contradicts Lemma~\ref{syzf}. So we must have that $\syz_R(Y)= 0$. Hence $N\cong F$.

We first notice from \ref{Hom-tensor-C} that $(M^*)^{\oplus m}\cong C^{\oplus n}$ for some $m,n\ge 1$. Since $\depth(R\oplus C)\ge t$, Proposition~\ref{semid} and Remark~\ref{remrf}(1) yield that $C$ is strongly $t$-residually faithful, and hence $(M^*)^{\oplus m}\cong C^{\oplus n}$ is so. Consequently, $M^*$ is strongly $t$-residually faithful. On the other hand, since $\Ext_R^{1\le i\le t-1}(M,R)=0$, in view of \ref{para:Ext-syz}, $M^* \cong H\oplus \syz_R^{t+1}(L)$ for some $R$-modules $H$ and $L$ with $H$ free. By Lemma~\ref{syzf}, $\syz^{t+1}_R(L)$ is not strongly $t$-residually faithful. Therefore $H$ must be a nonzero module. In particular, $R$ is a direct summand of $M^*$. Thus, since $(M^*)^{\oplus m}\cong C^{\oplus n}$, one obtains that $R$ is a direct summand of $C^{\oplus n}$. Hence, by \cite[Lem.~1.2(i)]{LW}, $R$ is a direct summand of $C$. As $C$ is indecomposable, it follows that $C\cong R$. Therefore, from the given hypotheses, $\pd_R(M^*)<\infty$. Finally, Proposition~\ref{trt2} yields that $M\cong t(M)\oplus G$ for some free $R$-module $G$. Then $M^* \cong (t(M))^* \oplus G^* \cong G$. Since $\Hom_R(M,N)\neq 0$ and $N$ is nonzero free, it follows that $G$ is nonzero.

The remark in parenthesis about the particular cases is clear from \ref{para:-n-torsion-free-nth-syz} and \ref{para:G-dim-n-torsion-free}. 
\end{proof}  



\section{Finite projective dimension of Hom and vanishing of Ext}\label{sec:pd-hom}

In this section, we obtain various criteria for a module to be free in terms of vanishing of certain Ext modules and finite projective dimension of certain Hom. We start with the following remark which is often useful in this section.

\begin{remark}\label{22}
	If $M$ and $N$ are nonzero modules such that $\Ext_R^i(M,N)=0$ for all $1 \le i \le \depth(N)$, then $\Hom_R(M,N)$ is nonzero. This follows from the fact that $$\inf\{ n : \Ext_R^n(M,N)\neq 0 \} = \depth(\ann_R(M),N) \le \depth(N)<\infty,$$
	where the second inequality holds true because $\ann_R(M)$ is a proper ideal.
\end{remark}

The theorem below is motivated by Conjecture~\ref{Tachikawa}. It particularly provides a few criteria for a semidualizing module to be free.

\begin{theorem}\label{thm:pd-Hom-Ext-Tr}
Let $M$ and $N$ be nonzero modules over a local ring $R$. Set $t:=\depth(R)$ and $s:=\depth(N)$. Suppose that $\Ext_R^j(M,N)=0$ for all $1\le j\le s$, and $\pd_R(\Hom_R(M,N))<\infty$. Then, $s\le t$. Moreover, $M$ is free if at least one of the following holds true.
\begin{enumerate}[\rm (1)]
    \item $\Ext_R^i(\Tr M,R)=0$ for all $1\le i\le t+1$.   
    \item $\Ext_R^i(M^*,R)=0$ for all $1\le i\le t-1$, and $M$ is reflexive.
    \item $\Ext_R^i(\Tr M,R)=0$ for all $n+2 \le i\le t+1$, the module $M$ satisfies $(\widetilde S_{n+1})$, and locally $M$ has finite G-dimension in co-depth $n$ for some integer $n\ge 0$, {\rm(cf.~\ref{para-Sn} and \ref{para:property-in-codim})}.
    \item $\Ext_R^i(M^*,R)=0$ for all $n \le i\le t-1$, the module $M$ satisfies $(\widetilde S_{n+1})$, and locally $M$ has finite G-dimension in co-depth $n$ for some integer $n\ge 1$.
    \item $R$ is CM, $M$ is MCM and locally free on the punctured spectrum of $R$, and $\Ext_R^{t+1}(\Tr M,R)=0$.  
\end{enumerate}
\end{theorem}  

\begin{proof}
    Since $\Ext_R^j(M,N)=0$ for all $1\le j\le s$, by Remark~\ref{22}, $\Hom_R(M,N)\neq 0$. Moreover, Lemma~\ref{extres} yields that $\depth(\Hom_R(M,N)) \ge s$. Hence $0\le \pd_R(\Hom_R(M,N)) \le t-s$. So $s \le t$.

    (1) Since $ \Ext_R^j(\Tr M, R) = 0 $ for all $1\le j \le t+1 = s+(t-s)+1$, in view of Lemma~\ref{2.3}.(2), it follows that $\Ext_R^j(\Tr M, \Hom_R(M,N))=0$ for all $1\le j \le s+1$. Hence, by Proposition~\ref{propnew}, $M$ is free.
    
    (2) Note that $M$ is reflexive \iff $\Ext_R^i(\Tr M,R)=0$ for both $i=1$ and $2$, see, e.g., \cite[1.4.21]{BH93}. Since $M^*$ and $\Omega^2(\Tr M)$ are stably isomorphic (cf.~\ref{para:Tr-M}), it follows that
    \begin{equation}\label{eqn:Ext-M-star-Tr-M}
        \Ext_R^i(M^*,R) \cong \Ext_R^i(\Omega^2(\Tr M),R) \cong \Ext_R^{i+2}(\Tr M,R) \; \mbox{ for all }i \ge 1.
    \end{equation}
    Thus the conditions in (2) imply that $\Ext_R^{1\le i\le t+1}(\Tr M,R)=0$, and hence by (1), $M$ is free. 

(3) Let $ n \ge 0 $. Since $M$ satisfies $(\widetilde S_{n+1})$, and locally $M$ has finite G-dimension in co-depth $n$, by \cite[Prop.~2.4.(b)]{DS15}, $\Ext_R^{1\le i\le n+1}(\Tr M,R)=0$. Combining with the other vanishing conditions, $\Ext_R^{1\le i\le t+1}(\Tr M,R)=0$. Hence it follows from (1) that $M$ is free.

(4) Let $n \ge 1$. As in (3), one obtains that $\Ext_R^{1\le i\le n+1}(\Tr M,R)=0$. Since $\Ext_R^{n \le i\le t-1}(M^*,R)=0$, in view of \eqref{eqn:Ext-M-star-Tr-M}, $\Ext_R^{n+2 \le i\le t+1}(\Tr M,R)=0$. Thus $\Ext_R^{1 \le i\le t+1}(\Tr M,R)=0$, and hence $M$ is free.

(5) If $t=0$, since $\Ext_R^1(\Tr M,R)=0$, by (1), $M$ is free. So we may assume that $t\ge 1$. Since $R$ is CM, and $M$ is MCM, it is well known that $M_\fp$ is an MCM $R_\fp$-module for every prime ideal $\fp\in\Supp(M)$, see, e.g., \cite[2.1.3.(b)]{BH93}. Thus, in particular, $M$ satisfies $(\widetilde S_t)$. Since $M$ is locally free on the punctured spectrum of a CM local ring $R$, one obtains that $M$ is locally free in co-depth $t-1$. Therefore the condition (3) is satisfied for $n=t-1$. Hence $M$ is free.
\end{proof}



As a corollary of Theorem~\ref{thm:pd-Hom-Ext-Tr}, we obtain the following characterizations of Gorenstein local rings.

\begin{corollary}\label{cor:Gor-rings-Tr-omega}
    Let $(R,\fm)$ be a CM local ring of dimension $d$ with a canonical module $\omega$. Then the following statements are equivalent:
    \begin{enumerate}[\rm (1)]
        \item $R$ is Gorenstein.
        \item $\Ext_R^{1\le i\le d+1}(\Tr(\omega),R)=0$.
        \item $\Ext_R^{1\le i\le d-1}(\omega^*,R)=0$, and $\omega$ is reflexive.
        \item $\Ext_R^{n+2 \le i\le d+1}(\Tr(\omega),R)=0$, and $R$ is locally Gorenstein in codimension $n$ for some $0\le n \le d-1$.
        \item $\Ext_R^{n\le i\le d-1}(\omega^*,R)=0$, and $R$ is locally Gorenstein in codimension $n$ for some $1\le n \le d-1$.
        \item $\Ext_R^{d+1}(\Tr(\omega),R)=0$, and $R$ is locally Gorenstein on $\Spec(R)\smallsetminus \{\fm\}$.
    \end{enumerate}
\end{corollary}

\begin{proof}
    Note that $R$ is Gorenstein \iff $\omega$ is free. So (1) trivially implies all the other conditions. The reverse implications can be deduced from Theorem~\ref{thm:pd-Hom-Ext-Tr} by setting $M=N=\omega$ as $\Hom_R(\omega,\omega)\cong R$, $\Ext_R^j(\omega,\omega)=0$ for all $j\ge 1$, and $\depth(\omega)=\depth(R)=d$. Note that for $\fp\in\Spec(R)$, if $R_\fp$ is Gorenstein, then $\omega_\fp \cong R_\fp$.
\end{proof}

\begin{remark}
We note here that the main non-trivial implication (2) $\Rightarrow$ (1) of
Corollary~\ref{cor:Gor-rings-Tr-omega} can also be proved by more classical methods as follows: The hypothesis of $(2)$ implies that $\omega$ is $(d+1)$-torsion-free, hence a $(d+1)$-st syzygy (cf.~\ref{para:-n-torsion-free-nth-syz}). So we have an exact sequence $0\to \omega \to F \to L\to 0$, where $F$ is free and $L$ is a $d$-th syzygy module, hence $L$ is MCM, which implies that $\Ext^1_R(L,\omega)=0$. Consequently, the exact sequence splits, giving $\omega$ is free, i.e., $R$ is Gorenstein.  
\end{remark}

\begin{remark}
The equivalences (1) $\Leftrightarrow$ (2) $\Leftrightarrow$ (6) in Corollary~\ref{cor:Gor-rings-Tr-omega} recover \cite[Cor.~5.19 and 5.20]{HM22}, while the implication (5) $\Rightarrow$ (1) strengthens \cite[Cor.~5.22.(iv)$\Rightarrow$(i)]{HM22} as we do not assume the vanishing of $\Ext_R^d(\omega^*,R)$.  
\end{remark}

Now we prove the 2nd main result of this section. Note that Theorem~\ref{thm:Hom-pd-id}.\eqref{Hom-Ext-revisited-thm-3.8} considerably strengthens \cite[Thm.~3.8]{DEL21}.

\begin{theorem}\label{thm:Hom-pd-id}
Let $(R,\m)$ be a local ring of depth $t$. Let $M$ and $N$ be $R$-modules such that $\Hom_R(M,N)$ is nonzero and $ \Ext_R^{1\le i \le t-1}(M,N) = 0 $. Suppose $N \cong F \oplus \syz_R(Y)$ for some $R$-module $Y$ of depth at least $t$ and free $R$-module $F$ $($e.g., when $N$ is $(t+1)$-torsion-free, or $\gdim_R(N)=0$$)$.
\begin{enumerate}[\rm (1)]
    \item \label{Hom-pd-N-special}
    If $\pd_R(\Hom_R(M,N)) < \infty $, then $N\cong F$ and $M\cong t(M)\oplus R^{\oplus s}$ for some $s\ge 1$.
    \item \label{Hom-Ext-revisited-thm-3.8}
    If $ \Ext_R^t(M,R) = 0 $ and $ \pd_R(\Hom_R(M,N)) < \infty $, then both $M$ and $N$ are free.
    \item \label{id-hom-3}
    If $\id_R(\Hom_R(M,N)) < \infty $, then $R$ is Gorenstein, $N \cong F$ and $M \cong \Gamma_\fm(M) \oplus R^{\oplus r}  \cong t(M)\oplus R^{\oplus s}$ for some $r\ge 0$ and $s\ge 1$.
    \item \label{id-hom-4}
    If $ \Ext_R^t(M,R) = 0 $ and $ \id_R(\Hom_R(M,N)) < \infty $, then both $M$ and $N$ are free.
\end{enumerate}
\end{theorem}

\begin{proof}
    (1) With $C=R$, the claim follows from Theorem\ref{thm:pd-Hom-finite-gdim-freeness-criteria}.

    (2) In view of (1), $N\cong F$ and $M\cong t(M)\oplus R^{\oplus s}$ for some $s\ge 1$. Hence, from the given hypotheses, one obtains that $\Ext_R^{1\le i \le t}(t(M),R)=0$. Note that $\Hom_R(t(M),R)=0$. Therefore, by Remark~\ref{22}, $t(M)$ must be zero. Thus both $M$ and $N$ are free.
    
    (3) Since $\Hom_R(M,N)$ is nonzero and of finite injective dimension, by Bass' conjecture, $R$ is CM. First we prove that $N \cong F$. For that, without loss of generality, we may pass to the completion, and assume that $R$ admits a canonical module $\omega$. As $\Hom_R(M,N)$ has finite injective dimension, $\Hom_R(\omega, \Hom_R(M,N))$ has finite projective dimension, see, e.g., \cite[9.6.5]{BH93}. Therefore, since $\omega$ is semidualizing, by virtue of Theorem\ref{thm:pd-Hom-finite-gdim-freeness-criteria}, $N \cong F$ and $\omega\cong R$. In particular, $R$ is Gorenstein. Moreover, it follows that $M^*$ is nonzero, $\Ext_R^{1\le i \le t-1}(M,R)=0$ and $\id_R(M^*)$ is finite. Hence, by \cite[Cor.~2.10.(1)]{GT21}, $M \cong \Gamma_\fm(M) \oplus R^{\oplus r}$ for some $r\ge 0$. Since $R$ is Gorenstein, $\Hom_R(M,N)$ has finite projective dimension as well. So the other isomorphism follows from (1).

    (4) In view of (3), $N \cong F$ and $M \cong t(M)\oplus R^{\oplus s}$ for some $s\ge 1$. Similar arguments as in the proof of (2) yield that $M$ is free as well.
\end{proof}  

We have a number of consequences of Theorem~\ref{thm:Hom-pd-id}.  

\begin{corollary}\label{cor:Hom-pd-Ext-vanishing-N-special}
Let $R$ be a local ring of depth $t$. Let $M$ and $N$ be nonzero $R$-modules such that $\pd_R(\Hom_R(M,N))<\infty$ or $\id_R(\Hom_R(M,N))<\infty$. Also assume that $\Ext_R^i(M,N)=0$ for all $1 \le i \le \depth(N)$. Furthermore, let $N \cong F \oplus \syz_R(Y)$ for some $R$-module $Y$ of depth at least $t$ and free $R$-module $F$ $($e.g., when $N$ is $(t+1)$-torsion-free, or $\gdim_R(N)=0$$)$. Then, both $M$ and $N$ are free.
\end{corollary}

\begin{proof}
In view of Remark~\ref{22}, $\Hom_R(M,N)$ must be nonzero. Since $\depth(N) \ge t$, by the given hypotheses, $\Ext_R^i(M,N)=0$ for all $1 \le i \le t$. So Theorem~\ref{thm:Hom-pd-id}.\eqref{Hom-pd-N-special} and \eqref{id-hom-3} yield that $N$ is free. Hence, by Theorem~\ref{thm:Hom-pd-id}.\eqref{Hom-Ext-revisited-thm-3.8} and \eqref{id-hom-4}, $M$ is free as well.
\end{proof}

As another consequence of Theorem~\ref{thm:Hom-pd-id}, we also confirm Conjecture~\ref{ARC} under the condition that both $\gdim_R(M)$ and $\pd_R(\Hom_R(M,M))$ are finite.

\begin{corollary}\label{cor:pd-Hom-G-dim-M-finite}
	Let $M$ be a module over a local ring $R$ such that both $\gdim_R(M)$ and $\pd_R(\Hom_R(M,M))$ are finite. Set $t:=\depth(R)$. Assume that
    \[
        \Ext_R^i(M,R)=\Ext_R^j(M,M)=0 \mbox{ for all } 1 \le i \le t \mbox{ and } 1 \le j \le t-1.
    \]
    Then $M$ is free.
\end{corollary}

\begin{proof}
	Since $\gdim_R(M)$ is finite, and $\Ext_R^i(M,R)=0$ for all $1 \le i \le t$, it follows that $\gdim_R(M) = 0$. Therefore, by virtue of Theorem~\ref{thm:Hom-pd-id}.\eqref{Hom-Ext-revisited-thm-3.8}, $M$ is free.
\end{proof}

Localizing at each prime ideal and using Proposition~\ref{trt2}.(2) and Corollary~\ref{cor:pd-Hom-G-dim-M-finite}, we immediately obtain the following consequence.

\begin{corollary}\label{cor:ARC-pd-Hom}
	The Auslander-Reiten conjecture holds true for a module $M$ over a commutative Noetherian ring $R$ in each of the following cases:
	\begin{enumerate}[\rm (1)]
		\item $ \pd_R\left( \Hom_R(M,R) \right) < \infty $.
		\item $ \pd_R\left( \Hom_R(M,M) \right) < \infty $ and $\gdim_R(M) < \infty$.
	\end{enumerate}
\end{corollary}


\section{Applications to $n$-semidualizing modules and divisor class group}\label{clgroup}
   
In this section, we provide a complete answer to Question \ref{ques1:Tony-Se} by proving a much more general statement. Moreover, under mild additional hypothesis, we affirmatively answer  Question~\ref{ques2:Tony-Se}. Furthermore, for an integral domain satisfying $(S_3)$ and $(R_2)$, we show that the divisor class group coincides with $1$-semidualizing modules (Theorem~\ref{thm:S1-Cl}.(2)); since this class of rings contains all determinantal rings with coefficients in a regular local ring, our result (Corollary \ref{cor:on-deter-rings}) highly generalizes \cite[Thm.~4.13]{Se}, the main result in \cite{Se}. Recall that the set of all isomorphism classes of $n$-semidualizing modules of $R$ is denoted by $\fS^n_0(R)$, cf.~\ref{defn:n-semidualizing}.

\begin{para}\label{para:class-group}\cite[pp.~261-262]{Sa07}
    Let $R$ be a normal domain. The divisor class group of $R$, denoted $\cl(R)$, is defined to be the set of all isomorphism classes $[M]$ of reflexive $R$-modules $M$ of rank $1$. It forms an abelian group with the group operations
    \begin{equation}\label{eqn:Cl-gp-operations}
        [M]+[N] = [(M\otimes_R N)^{**}] \; \mbox{ and } \; [M]-[N] = [\Hom_R(N,M)].
    \end{equation}
    Note that the identity element of $\cl(R)$ is $[R]$.
\end{para}

\begin{para}\label{para:Cl-subset-S0}
    Let $R$ be a normal domain. Then every $0$-semidualizing module has rank $1$ (here $R$ is not required to be normal). Conversely, every module of rank $1$ is $0$-semidualizing, see \cite[Lem.~3.4]{Se}. Thus $\cl(R) \subseteq \fS^0_0(R)$.
\end{para}

\begin{para}\label{para:S1-subset-cl-R}
    It is shown in \cite[Prop.~3.6]{Se} that if $R$ is a normal domain, then we have a set inclusion $\fS_0^1(R) \subseteq \cl(R)$. Both the sets are identical for certain Gorenstein determinantal rings over a field, see \cite[Thm.~4.13]{Se}. However, there are Gorenstein normal domains for which $\fS_0^1(R) \neq \cl(R)$, due to \cite[Ex.~5.1]{Se}.
\end{para}   

The following lemma, which gives kind of a converse to Lemma \ref{extres} when $M=N$, is proved in \cite[Prop.~2.2.(a)]{Se} under the additional condition that $\Hom_R(M,M)\cong R$.   
  
\begin{lemma}\label{semi}
Let $M$ be a module over a local ring $R$ such that $\Ext^{1\le i\le n-1}_R(M,M)=0$ for some integer $n\ge 1$. If a sequence $x_1,\dots, x_n$ is $R$-regular and $\Hom_R(M,M)$-regular, then it is also $M$-regular.
\end{lemma}     

\begin{proof}
We prove this by induction on $n$. For $n=1$, we have no Ext vanishing. Since $\Ass_R(\Hom_R(M,M))=\Ass_R(M)$, the set of zero-divisors of $M$ and $\Hom_R(M,M)$ coincide, which proves the $n=1$ case.  Now assume that $n\ge 2$, and $x_1,\dots,x_n$ is a regular sequence over both $R$ and $\Hom_R(M,M)$. By the base case, $x_1$ is $M$-regular, so we have an exact sequence $0\to M \xrightarrow{x_1\cdot} M \to M/x_1 M \to 0$. Applying $\Hom_R(M,-)$ to this sequence, since $\Ext^{1\le i\le n-1}_R(M,M)=0$ (particularly, $\Ext^1_R(M,M)=0$ as $n-1\ge 1$), we get another exact sequence  
\begin{equation}\label{eqn:ses-hom}
    0\to \Hom_R(M,M) \xrightarrow{x_1\cdot} \Hom_R(M,M)\to \Hom_R(M,M/x_1 M)\to 0
\end{equation}
and $\Ext_R^{1\le i \le n-2}(M,M/x_1M)=0$. Hence, in view of \cite[pp.~140, Lem.~2.(ii)]{Mat86}, one obtains that $\Ext_{R/x_1R}^{1\le i \le n-2}(M/x_1M,M/x_1M)=0$ and
$$\dfrac{\Hom_R(M,M)}{x_1\Hom_R(M,M)} \cong \Hom_R(M,M/x_1 M)\cong \Hom_{R/x_1R}(M/x_1M,M/x_1M).$$
Therefore, since $x_2,\dots,x_n$ is regular over both $R/x_1R$ and $$\Hom_R(M,M)/x_1\Hom_R(M,M)\cong \Hom_{R/x_1R}(M/x_1M,M/x_1M),$$ by induction hypothesis, it is also regular on $M/x_1 M$, hence $x_1,\dots,x_n$ is an $M$-regular sequence. 
\end{proof}     

Using Lemma~\ref{semi}, we show that over a local ring $R$, an $(n-1)$-semidualizing module has depth at least $\inf \{n, \depth R\}$.

\begin{proposition}\label{prop:n-1-semidual-Sn}
Let $R$ be a local ring, and $M \neq 0$ be an $R$-module. Let $n\ge 1$ be an integer such that $\Ext^{1\le i\le n-1}_R(M,M)=0$, and $\pd_R(\Hom_R(M,M))<\infty$. Then
$$\depth_R(M) \ge \inf \{n, \depth_R(\Hom_R(M,M))\}.$$
\end{proposition} 

\begin{proof}
Let $m=\inf \{n, \depth_R(\Hom_R(M,M)) \}$. Consider a regular sequence ${\bf x} = x_1,\ldots,x_m$ on $\Hom_R(M,M)$. Since $\Hom_R(M,M)\neq 0$ has finite projective dimension, by the Auslander's zero-divisor conjecture (cf.~\cite[9.4.7 and 9.4.8.(a)]{BH93}), ${\bf x}$ is also $R$-regular.
Hence, by Lemma~\ref{semi}, ${\bf x}$ is $M$-regular. Thus $\depth_R(M) \ge m$.
\end{proof}

For a ring, which is not necessarily local, we  have the following result.

\begin{corollary}\label{serhom}
Let $R$ be a ring, and $M$ be an $R$-module. Suppose there is an integer $n\ge 1$ such that $\Ext^{1\le i\le n-1}_R(M,M)=0$.
\begin{enumerate}[\rm (1)]
    \item
    If $\pd_R(\Hom_R(M,M))<\infty$, then
    $$\depth_{R_{\p}}(M_{\p}) \ge \inf\{n, \depth_{R_{\p}}\big(\Hom_R(M,M)_{\p}\big)\} \text{ for all } \p \in \Spec(R).$$
    \item 
    If $\Hom_R(M,M)$ is $R$-projective, then $M$ satisfies $(\widetilde S_n)$ $($\rm cf.~\ref{para-Sn}$)$.
\end{enumerate}
\end{corollary}  

\begin{proof}
    (1) Let $\p \in \Spec(R)$. If $M_{\p}=0$, then $\depth_{R_{\p}}(M_{\p})=\infty$, and there is nothing to show. When $M_{\p}\neq 0$, the desired inequality follows from Proposition~\ref{prop:n-1-semidual-Sn}.

    (2) It is obvious by (1) as every projective module over $R$ is locally free.
\end{proof}

The theorem below particularly shows that over a local ring $R$ of depth $t$, every $(t-1)$-semidualizing module of finite G-dimension is free.

\begin{theorem}\label{mainsemi}
Let $R$ be a local ring of depth $t$. Let $M$ be a nonzero $R$-module such that $\Ext^{1\le i\le t-1}_R(M,M)=0$ and $\gdim_R(M) <\infty$. If $\Hom_R(M,M) \cong C^{\oplus n}$ for some semidualizing module $C$, then $C\cong R$, and $M$ is free of rank $r$ satisfying $r^2=n$. In particular, if $\Hom_R(M,M) \cong R$, then $M\cong R$.
\end{theorem}

\begin{proof}
From the given hypothesis on $\Hom_R(M,M)$, one has that $\depth_R(\Hom_R(M,M)) =\depth(C) =\depth(R)$. Consider an $R$-regular sequence ${\bf x} := x_1,\ldots,x_t$. Since $C$ is semidualizing, $\bf x$ is regular on $C$, and hence regular on $\Hom_R(M,M)$ as well. So, by Lemma~\ref{semi}, $\bf x$ is also regular on $M$. Thus $\depth(M) \ge t$. Therefore, since $\gdim_R(M)<\infty$, by Auslander-Bridger formula, $\gdim_R(M)=0$. Note that $\Hom_R(C,\Hom_R(M,M)) \cong R^{\oplus n}$ is a free $R$-module. So, by virtue of Theorem~\ref{thm:pd-Hom-finite-gdim-freeness-criteria}, $M$ is free, i.e., $M\cong R^{\oplus r}$ for some $r\ge 1$. Consequently, $\Hom_R(R^{\oplus r},R^{\oplus r}) \cong C^{\oplus n}$, which implies that $C$ is free, and hence $C\cong R$ (as $C$ is semidualizing). Comparing rank of the free modules, one obtains that $r^2=n$. The last part follows from the first one.
\end{proof} 

 As a consequence of Theorem~\ref{mainsemi}, we get the following result, which particularly gives a complete affirmative answer to Question~\ref{ques1:Tony-Se}.  

\begin{corollary}\label{qa}
Let $R$ be a ring $($not necessarily local$)$ of finite dimension $d$. Let $M$ be an $R$-module such that $\Hom_R(M,M)$ is $R$-projective and $\Ext^{1\le i\le d-1}_R(M,M) = 0$. The following statements hold true.
\begin{enumerate}[\rm (1)]
    \item If $\gdim_R(M)<\infty$, then $M$ is projective.
    \item If $R$ is Gorenstein, then $M$ is projective. Particularly, if $R$ is Gorenstein, then every $(d-1)$-semidualizing $R$-module is in fact semidualizing.
\end{enumerate}
\end{corollary}

\begin{proof}
    (1) Let $\p \in \Spec(R)$. Localizing the given conditions, $\Hom_{R_{\p}}(M_{\p},M_{\p})$ is $R_{\p}$-free, $\Ext^{1\le i\le d-1}_{R_{\p}}(M_{\p},M_{\p}) = 0$ and $\gdim_{R_{\p}}(M_{\p})<\infty$. Therefore, since $\depth(R_{\p}) \le \dim(R_{\p}) \le d$, by Theorem~\ref{mainsemi}, $M_{\p}$ is free. Thus $M$ is projective.

    (2) If $R$ is Gorenstein, then $\gdim_R(M)<\infty$, and hence $M$ is projective by (1). The last part follows because a projective $R$-module $M$ satisfying $\Hom_R(M,M)\cong R$ is in fact semidualizing.
\end{proof}

\begin{remark}
    We got to know from Mohsen Asgharzadeh that \Cref{qa}.(2) is shown in \cite[Prop.~8.6]{Asg18} for Gorenstein local rings. However, \Cref{qa}.(1) highly strengthens \cite[Prop.~8.6]{Asg18}.
\end{remark}

In the context of \ref{para:property-in-codim}, the following corollary is immediate from Corollary~\ref{qa}.(2).

\begin{corollary}\label{35}
Let $R$ be a ring locally Gorenstein in codimension $n$. Let $M$ be an $R$-module such that $\Hom_R(M,M)$ is $R$-projective and $\Ext^{1\le i\le n-1}_R(M,M)=0$. Then $M$ is locally free in codimension $n$. 
\end{corollary}  




Since the proof of \cite[Lem.~2.3.(1)$\Rightarrow$(2)]{Da10} does not require the ring to be CM, and both the hypothesis and the conclusion of the result localizes, we get the following lemma. 

\begin{lemma}\cite[Lem.~2.3]{Da10}\label{dao}
Let $R$ be a ring $($not necessarily local$)$, and $M,N$ be $R$-modules. Fix an integer $n>1$. Suppose that
\begin{enumerate}[\rm (1)]
    \item $M$ is locally free in codimension $n$,
    \item $N$ satisfies $(S_n)$, and $\Hom_R(M,N)$ satisfies $(S_{n+1})$.
\end{enumerate}
Then $\Ext_R^i(M,N)=0$ for all $1 \le i \le n-1$.
\end{lemma}

The following result gives a partial affirmative answer to Question~\ref{ques2:Tony-Se}.

\begin{theorem}\label{thm:S1-Cl}
Let $R$ be a normal domain satisfying $(S_3)$.
\begin{enumerate}[\rm (1)]
    \item If $R$ is locally Gorenstein in codimension $2$, then $\fS_0^1(R)$ is a subgroup of $\cl(R)$.
    \item If $R$ satisfies $(R_2)$, then $ \fS^1_0(R) = \cl(R) $.
\end{enumerate}
\end{theorem}

\begin{proof}
    Since $R$ is a normal domain, in view of \ref{para:class-group} and \ref{para:S1-subset-cl-R}, we already have that $\cl(R)$ is an abelian group, and $\fS_0^1(R)$ is a subset of $ \cl(R) $.

    (1) Clearly, $\fS_0^1(R)$ contains $[R]$, the identity element of $\cl(R)$. Let $M,N \in \fS_0^1(R)$. Then $ [N]-[M] = [\Hom_R(M,N)] \in \cl(R)$. We need to show that $ [\Hom_R(M,N)] \in \fS^1_0(R)$, i.e., $\Hom_R(M,N)$ is $1$-semidualizing. Since $ \cl(R) \subseteq \fS^0_0(R) $ (by \ref{para:Cl-subset-S0}), $\Hom_R(M,N)\in \fS^0_0(R)$. So it is enough to prove that
\begin{equation}\label{eqn:Ext1-vanishing}
    \Ext^1_R\big( \Hom_R(M,N),\ \Hom_R(M,N) \big) = 0.
\end{equation}
Since $M,N\in \fS^1_0(R)$, by Corollary~\ref{35}, both $M$ and $N$ are locally free in codimension $2$. So $\Hom_R(M,N)$ is locally free in codimension $2$. As $N\in \fS^1_0(R)$, by Corollary~\ref{serhom}.(2), $N$ satisfies $(\widetilde S_2)$, and hence $\Hom_R(M,N)$ satisfies $(\widetilde S_2)$, see, e.g., \cite[1.4.19]{BH93}. It follows that $\Hom_R(M,N)$ satisfies $(S_2)$ as $R$ satisfies $(S_2)$.  
Since $\Hom_R(M,N)\in \fS^0_0(R)$, $\Hom_R(\Hom_R(M,N),\Hom_R(M,N))\cong R$, which satisfies $(S_3)$. Therefore the equality \eqref{eqn:Ext1-vanishing} follows from Lemma~\ref{dao}. 

(2) We already have $\fS_0^1(R) \subseteq \cl(R)$. To prove the other inclusion, let $[M] \in \cl(R)$. Then, by the definition, $M$ is a reflexive module of rank $1$. Hence $M$ is a second syzygy module. So, by the depth lemma, $M$ satisfies $(\widetilde S_2)$, and hence $(S_2)$ as $R$ satisfies $(S_2)$. Thus $\depth_{R_{\p}}(M_{\p})\ge \dim(R_{\p})$ for all $\p \in \Spec(R)$ with $\dim(R_{\p})\le 2$. Since $R$ is regular in codimension $2$, by the Auslander-Buchsbaum formula, it follows that $M$ is locally free in codimension $2$. Since $[M]\in \cl(R) \subseteq \fS^0_0(R)$ (cf.~\ref{para:Cl-subset-S0}), $\Hom_R(M,M)\cong R$, which satisfies $(S_3)$. So, by Lemma~\ref{dao}, $\Ext^1_R(M,M)=0$. Thus $[M] \in \fS^1_0(R)$.
\end{proof}

\begin{para}\label{para:Deter-rings}
    Let $B$ be a commutative Noetherian ring. Let $X$ be an $m\times n$ matrix of indeterminates over $B$.
    Set $B[X] := B[X_{ij}:1\le i\le m, 1\le j\le n]$, where $X_{ij}$ are the entries of $X$. 
    Consider an integer $t$ satisfying $1\le t \le \min\{m,n\}$. Let $I_t(X)$ be the ideal of $B[X]$ generated by the $t\times t$ minors of $X$. The quotient ring $R_t(X) := B[X]/I_t(X)$ is called the determinantal ring of $t$-minors of $X$ over $B$. The following properties hold for determinantal rings.    
    \begin{enumerate}[(1)]
        \item 
        If $B$ is a (normal) domain, then so is $R_t(X)$, see, e.g., \cite[5.3.(a) and 6.3]{BV}.
        \item 
        If $B$ is CM, by \cite[Cor.~4]{HE}, it follows that $R_t(X)$ is also CM.
        \item 
        If $B$ satisfies $(R_2)$, then so is $R_t(X)$, see, e.g., \cite[5.3.(a) and 6.12]{BV}.
    \end{enumerate}
\end{para}

Since CM domains satisfying $(R_2)$ are normal, the following corollary is immediate from \ref{para:Deter-rings} and Theorem~\ref{thm:S1-Cl}.(2).

\begin{corollary}\label{cor:on-deter-rings}
    Let $B$ be a CM domain satisfying $(R_2)$. Then, the determinantal ring $R_t(X)= B[X]/I_t(X)$ is  also a CM domain satisfying $(R_2)$, consequently, $ \fS^1_0(R_t(X)) = \cl(R_t(X)) $.
\end{corollary}

\end{document}